\documentclass[12pt]{iopart}

\usepackage[T1]{fontenc} 
\usepackage[utf8]{inputenc}

\usepackage{amsfonts,pifont,amssymb,latexsym}
\usepackage{chemarrow}
\usepackage{amsthm,amsopn}
\usepackage[english]{babel}

\usepackage{tikz} 
\usetikzlibrary{patterns}
\usepackage[pdftex]{hyperref}

\tikzstyle DDS=[black,->]
\tikzstyle TDDS=[gray,double,->]
\tikzstyle PCA=[black,dotted,->]
\tikzstyle PCAt=[gray,double,dotted,->]
\tikzstyle oSS=[black,dashed,->]
\tikzstyle soSS=[gray,double,dashed,->]
\tikzstyle eqCA=[rounded corners,draw=none,fill=black!8]

\usepackage{bbm} 

\newtheorem{thm}{Theorem}

\newtheorem{lem}{Lemma}

\newtheorem{prop}{Proposition}
\newtheorem{cor}[prop]{Corollary}

\newtheorem{open}{Open question}

\newtheorem{defi}{Definition}
\newtheorem{rem}{Remark}

\newtheorem{conj}{Conjecture}

\theoremstyle{plain}
\newtheorem{exmp}{Example}

\newcommand{\dfn}[1]{\emph{#1}}
\newcommand{\ssl}[1]{#1}
\newcommand{\xpr}[1]{``#1''}

\newcommand{\resp}[1]{\ (resp. #1)}
\newcommand{\ie}{\textit{i.e.}\ }

\newcommand{\soit}[1]{\left|\everymath{\displaystyle\everymath{}}\begin{array}{ll}#1\end{array}\right.}
\newcommand{\appli}[4]{\begin{array}{rcl}#1&\to&#2\\
#3&\mapsto&\displaystyle#4\end{array}}
\newcommand{\appl}[5]{#1:\appli{#2}{#3}{#4}{#5}}
\newcommand{\si}{\textrm{ if }}
\newcommand{\oub}{\textrm{ or }}

\newcommand{\impl}{\Rightarrow}
\newcommand{\imp}[3][:\ ]{\item[\ref{i:#2}$\impl$\ref{i:#3}#1]}

\newcommand\Zset{\ensuremath \mathbb{Z}}
\newcommand\Nset{\ensuremath \mathbb{N}}
\newcommand\Rset{\ensuremath \mathbb{R}}
\newcommand{\Mset}{\ensuremath \mathbb{M}}
\newcommand{\Ns}{\ensuremath \mathbb{N}_+}
\newcommand{\an}{A^\Nset} 
\newcommand{\anm}{A^{-\Nset}}
\newcommand{\az}{A^\Zset}
\newcommand{\am}{A^\Mset}
\newcommand{\trois}{\{0,1,2\}}

\newcommand{\length}[1]{\left|#1\right|} 
\newcommand{\restr}[1]{_{\left|#1\right.}} 
\newcommand{\cl}[1]{\overline{#1}} 
\newcommand{\ball}[2]{\mathcal B_{#1}(#2)} 
\newcommand{\card}[1]{\left|#1\right|} 
\newcommand{\abs}[1]{\left|#1\right|} 
 \newcommand{\spart}[1]{\left\lceil #1\right\rceil}
 
\newcommand{\dist}{\mathbbm d} 
\newcommand{\lang}{\mathcal L} 
\newcommand{\g}{\mathcal G} 

\newcommand{\sett}[2]{\left\{#1\mid#2\right\}}
\newcommand{\set}[3]{\sett{#1\in#2}{#3}} 

\newcommand{\tnd}[1][n]{\rightarrow_{#1\to\infty}{}} 

\usepackage{stmaryrd} 
\newcommand{\co}[2]{\left\llbracket #1,#2\right\llbracket}
\newcommand{\cc}[2]{\left\llbracket #1,#2\right\rrbracket}
\newcommand{\oo}[2]{\left\rrbracket #1,#2\right\llbracket}
\newcommand{\oc}[2]{\left\rrbracket #1,#2\right\rrbracket}
\newcommand{\ci}[1]{\co{#1}\infty}
\newcommand{\io}[1]{\oo{-\infty}{#1}}
\newcommand{\scc}[2]{_{\cc{#1}{#2}}}
\newcommand{\sco}[2]{_{\co{#1}{#2}}}
\newcommand{\soo}[2]{_{\oo{#1}{#2}}}
\newcommand{\soc}[2]{_{\oc{#1}{#2}}}
\newcommand{\sci}[1]{_{\ci{#1}}}
\newcommand{\sio}[1]{_{\io{#1}}}

\newcommand{\deux}{\mathbbm2}
\newcommand{\compl}[1]{{#1}^C}

\newcommand{\orb}{\mathcal O} 
 
\newcommand{\Jcal}{\mathcal J}

\newcommand{\pinf}[1]{\vphantom{#1}^\infty{#1}}
\newcommand{\uinf}[1]{#1^\infty}
\newcommand{\dinf}[1]{\vphantom{#1}^\infty{#1}^\infty}

\newcommand{\forb}[1]{\Sigma_{#1}} 

\newcommand{\radi}[1]{\langle #1\rangle} 
\newcommand{\isub}[1]{_{\radi{#1}}}
\newcommand{\iexp}[1]{^{\radi{#1}}}

\DeclareMathOperator*{\Min}{Min}
\DeclareMathOperator*{\id}{id}

\newcommand{\conc}{\oplus} 
\newcommand{\fac}{\sqsubset} 
\newcommand{\nfac}{\not\sqsubset} 

\newcommand{\suba}[1]{_{\begin{array}c#1\end{array}}} %

\newcommand{\start}[1]{\begin{#1}\hfill\nopagebreak[5]\begin{itemize}}
\newcommand{\finish}[1]{\popQED\popQED\end{itemize}\end{#1}}

\begin{document}

\title{Asymptotic behavior of dynamical systems and cellular automata}

\author{Pierre Guillon$^1$ and Gaétan Richard$^2$}


\address{$^1$ Department of Mathematics,
University of Turku,
FI-20014 Turku, Finland}
\address{$^2$ Greyc,
Université de Caen \& CNRS,
boulevard du Maréchal Juin, 14\,000 Caen, France}

\eads{\mailto{pguillon@univ-mlv.fr}, \mailto{grichard@info.unicaen.fr}}

\begin{abstract}
  We study discrete dynamical systems through the topological concepts of limit set, which consists of all points that can be reached arbitrarily late, and asymptotic set, which consists of all adhering values of orbits. In particular, we deal with the case when each of these are a singleton, or when the restriction of the system is periodic on them, and show that this is equivalent to some simple dynamics in the case of subshifts or cellular automata. Moreover, we deal with the stability of these properties with respect to some simulation notions.
\end{abstract}


\maketitle


\section*{Introduction}

Complex systems are made of a great number of entities interacting locally with each other in a fully deterministic way. 
However, in many cases, the global behavior of these systems is very complex and the only known way to understand their dynamics is by simulating them. These systems appear in many different fields such as biology, physics, chemistry, sociology, \ldots

In order to achieve links between the well-known local rule and the long-term dynamical properties of the system, one of the classical points of view is the study of attractors (see~\cite{Cosnard:1985}). To achieve formal results on those kinds of properties, some regular model is needed. Therefore, one often introduces some compact topology and some continuous self-map representing the evolution: discrete-time dynamical systems. Besides, if we require spatial homogeneity, we define the model of cellular automata (see \cite{VonNeumann:1966,Hedlund:1969}), which are composed of an infinite number of cells disposed on a line, and endowed with a state chosen among a finite alphabet.

We focus here on the long-term behavior of such formal systems. It can be first represented by the limit set, which consists of all the configurations that can appear after an arbitrarily long time (see~\cite{Hurd:1987,Culik:1989}). A more restrictive notion is the asymptotic set, composed of all the configurations close to which the system is passing infinitely often (see~\cite{acc,nilpeng}).

In this paper, we shall develop a selective review on newly achieved properties linking local behavior and properties of the limit or asymptotic set for dynamical systems in general, cellular automata or subshifts. The paper is divided as follows: after giving all necessary definitions about dynamical systems in Section \ref{s:sdd}, we shall first treat the case of limit set (Section \ref{s:ls}) then the case of asymptotic set (Section \ref{s:as}).

\section{Topological dynamics}\label{s:sdd}

We will note $\Ns=\Nset\setminus\{0\}$ and $\deux=\{0,1\}$. 

\subsection{Dynamical systems}

We model complex systems by discrete (topological) dynamical systems. Even though ``real life'' time is continuous and making it discrete can introduce artifacts, it can be observed that this restriction already exhibits a very complex behavior.
\begin{defi}\label{d:sdd}
A \dfn{discrete dynamical system} (DDS) is a pair $(X,F)$ (or simply $F$ when there is no confusion), where $X$ is a nonempty \ssl{compact} metric space and $F:X\to X$ a \ssl{continuous} function. 
\end{defi}

Let us denote $\dist$ the distance on $X$ and ${\ball\varepsilon x}=\set yX{\dist(x,y)<\varepsilon}$ the open ball of center $x\in X$ and radius $\varepsilon>0$.

A subset $Y\subseteq X$ is \dfn{$F$-invariant}\resp{\dfn{strongly $F$-invariant}} if $F(Y)\subseteq Y$\resp{$F(Y)=Y$}. Given such a set, we can define the restriction $F\restr Y$ of our dynamical system to it ; $(Y,F\restr Y)$ (or simply $(Y,F)$) is a \dfn{subsystem} of $(X,F)$ if besides $Y$ is \ssl{closed}.
A DDS $(X,F)$ is \dfn{minimal} if it does not contain any strict subsystem, \ie if any closed $F$-invariant $Y\subset X$ is either $\emptyset$ or $X$.

\paragraph{Simulation}

Let us introduce some order on the dynamics produced by these systems: the \ssl{simulation}. Intuitively, one says that one system simulates another one if the latter can be embedded into the former in a continuous way.

\begin{defi}
A \dfn{morphism} of a DDS $(X,F)$ into another $(Y,G)$ is a continuous function $\Phi:X\to Y$ such that $\Phi F=G\Phi$.  If the morphism is \ssl{surjective}, it is a \dfn{factor map}, $(Y,G)$ is called \dfn{factor} of $(X,F)$ and $(X,F)$ \dfn{extension} of $(Y,G)$. If the morphism is \ssl{bijective}, it is a \dfn{conjugacy}; $(X,F)$ and $(Y,G)$ are said \dfn{conjugate}.
\\
A \dfn{simulation} of \dfn{period} $n\in\Ns$ by \dfn{steps} of $n'\in\Ns$ by a DDS $(X,F)$ of another $(Y,G)$ is a \ssl{factor map} of some subsystem $(X',F^n)$ of $(X,F^n)$ into $(Y,G^{n'})$.
\end{defi}
Moreover, the simulation is \dfn{direct} if $n=1$, \dfn{total} if $n'=1$, \dfn{complete} if $X'=X$, \dfn{exact} if the factor map is actually a conjugacy. We say in these cases that $(X,F)$ \dfn{simulates} \dfn{directly}\resp{\dfn{totally}, \dfn{completely}, \dfn{exactly}} $(Y,G)$.


\begin{figure}[!htp]
  \centering
  \begin{tikzpicture}[auto]
    \node (sf) at (0,0) {$X$};
    \node (ef) at (0,-3) {$X$};
    \node (ef1) at (0,-1) {$X$};
    \node (ef2) at (0,-2) {$X$};
    \node (sff) at (2,0) {$X'$};
    \node (sf1) at (2,-1) {$X'$};
    \node (sf2) at (2,-2) {$X'$};
    \node (eff) at (2,-3) {$X'$};
    \node (sg) at (4,0) {$Y$};
    \node (sg1) at (4,-.75) {$Y$};
    \node (sg2) at (4,-1.5) {$Y$};
    \node (sg3) at (4,-2.25) {$Y$};
    \node (eg) at (4,-3) {$Y$};
    \draw[->] (sf) to node {$F$} (ef1);
    \draw[->] (ef2) to node {$F$} (ef);
    \draw[dotted] (ef1) -- (ef2);
    \draw[->] (sff) to node {$F$} (sf1);
    \draw[->] (sf2) to node {$F$} (eff);
    \draw[dotted] (sf1) -- (sf2);
    \draw[->] (sg) to node {$G$} (sg1);
    \draw[->] (sg1) to node {$G$} (sg2);
    \draw[->] (sg3) to node {$G$} (eg);
    \draw[dotted] (sg2) -- (sg3);
    \draw[<<-] (sg) to node {$\Phi$} (sff);
    \draw[<<-] (eg) to node {$\Phi$} (eff);
    \node at (1,0) {$\supseteq$};
    \node at (1,-3) {$\supseteq$};
    \end{tikzpicture}
  \caption{Simulation ($\Phi$ is surjective)}
  \label{fig:simul}
\end{figure}

\paragraph{Configurations}

In the rest of the paper, we shall focus on a specific case of dynamical systems: totally disconnected discrete dynamical systems (TDDS). Intuitively, they correspond to discretizing the space where the interacting basic objects live. From now on, $A$ will be some finite alphabet and $\Mset$ will stand either for $\Nset$ or for $\Zset$. In this context, a \dfn{configuration} of the space is an infinite sequence $x\in\am$ of letters.

The set $\am$ is endowed with the topology induced by the distance $\dist(x,y)= 2^{\min_{x_i\ne y_i}\abs i}$, which corresponds to the product topology of the discrete topology on $A$. From Tychonoff's theorem, the resulting space is compact.
It is known that totally disconnect sets are homeomorphic to subsets of $\am$; hence we will restrict our study of TDDS to the systems $(\Sigma,F)$ on a subspace $\Sigma\subset\am$ of configurations.

If $x\in\am$ is a configuration and $i,k\in\Mset$, we denote $x\sco ik\fac x$ the finite pattern $x_ix_{i+1}\ldots x_{k-1}$. The same notation holds for any kind of interval (including infinite ones). If $u\in A^*$ and $i\in\Mset$, then $[u]_i$ denotes the \dfn{cylinder} $\set x\am{x\sco i{i+\length u}=u}$.  If $k\in\Nset$, we note $\radi k=\set i\Mset{\abs i\le k}$. For instance, $A\iexp k$ denotes the set of patterns of length $k+1$ if $\Mset=\Nset$, of length $2k+1$ (indexed from $-k$ to $k$) if $\Mset=\Zset$.  If $u\in A\iexp k$, then $[u]$ denotes the \dfn{central cylinder} $\set x\am{x\isub k=u}$ of configurations that share $u$ as a central pattern. More generally, if $U\subset A\iexp k$, the $[U]$ denotes the cylinder $\set x\am{x\isub k\in U}$. We can also extend the cylinder notation in the following flavor: $x=\pinf0[u]\uinf0$ stands for the configuration filled entirely with $0$ except for the portion $x\isub k=u$; similarly, $\pinf0[u]_i=\set x{[u]_i}{\forall j<i,x_i=0}$, and so on. 

Given two configurations $x,y\in\az$ and some cell $i\in\Mset$, we define the \dfn{concatenation} $x\conc_iy$ as the configuration $z$ such that $z_k=x_k\si k<i,y_k\si k\ge i$.

If $0 \in A$, we say that a configuration $c \in \am$ is \dfn{$0$-finite} if it is equal to $0$ except for some finite number of elements. One easy remark is that the set of $0$-finite configurations is dense in $\am$. 

\paragraph{Subshifts}

A natural operation on configurations consists in \xpr{translating} it; this operation is called the \dfn{shift} $\sigma : \am \to \am$, defined as $\sigma(x)_i = x_{i+1}$.  Most of the time, we choose to study TDDS which are homogeneous and therefore commute with the shift. 

\begin{defi}
  A \dfn{onesided}\resp{\dfn{twosided}} \dfn{subshift} is a \ssl{closed} \ssl{$\sigma$-invariant}\resp{\ssl{strongly}} subset $\Sigma$ of $\an$\resp{$\az$}.
\end{defi}

Here, we introduce two classical different characterizations using either languages or graphs.

If $L\subset A^*$ is a language, then the set ${\forb L}=\set z\am{\forall u\in L,u\nfac z}$ of configurations avoiding patterns of $L$ is a subshift. Conversely, to any subshift $\Sigma$ can be associated the \dfn{language} ${\lang(\Sigma)}=\set u{A^*}{\exists z\in\Sigma,u\fac z}$ of the finite patterns appearing in some of its configurations. For any length $k\in\Nset$, the \dfn{language of order $k$} of $\Sigma$ is ${\lang_k(\Sigma)}=\lang(\Sigma)\cap A^k$.  A \dfn{forbidden language} of a subshift $\Sigma$ is a language $L\subset A^*$ such that $\Sigma=\forb L$.  A subshift is \dfn{of finite type} (SFT) if it admits some finite forbidden language. It is of \dfn{order $k\in\Nset$} ($k$-SFT) if it admits a forbidden language included in $A^k$.

Let us define a \dfn{graph} on alphabet $A$ as a pair ${\g=(V,E)}$ where $V$ is the finite set of \dfn{vertices}, $E\subset V\times V\times A$ the finite set of \dfn{arcs}; if $(v,w,a)\in E$ then $v$ is the \dfn{initial} vertex of the arc, $w$ its \dfn{terminal} vertex and $a$ its \dfn{label}.
A \dfn{path} is a sequence $(v_j,w_j,a_j)_{j\in I}\in E^\Mset$ of arcs where $I\subset\Mset$ and $v_{j+1}=w_j$ for $j,j+1\in I$. Its \dfn{label} is the sequence $(a_j)_{j\in I}$. A graph is \dfn{strongly connected} if any two vertices $v,w\in V$ belong to a same path. The \dfn{label system} of a graph $\g = (V,E)$ on $A$ is the subshift ${\Gamma_\g}=\sett{(a_j)_{j\in\Mset}}{(v_j,w_j,a_j)_{j\in\Mset}\in\Sigma_\g}$ of the \ssl{labels} of its infinite paths. A subshift is \dfn{sofic} if it is the \ssl{label system} of some graph. If we see the graph as a finite automaton, we can see that a subshift is sofic if and only if its language is regular.
Another equivalence is given by the Weiss' theorem \cite{weiss}: a subshift is sofic if and only if it is the factor of some SFT.

\paragraph{Cellular automata}

We introduce a kind of dynamical system which represent spatially homogeneous dynamics: cellular automata. Formally, a (one-dimensional) \dfn{cellular automaton} (CA) on alphabet $A$ is a triplet $(m,d,f)$ where $m \in \Mset$ is the \dfn{anchor}, $d \in \Nset$ is the \dfn{diameter} and $f: A^d \to A$ the \dfn{local transition rule}. We shall assimilate the cellular automaton with its associated dynamical system $(\am,F)$ defined by $F(x)_i=f(x \sco {i-m}{i-m+d})$ for any $x \in \am$ and any $i\in\Mset$, under the motivation of the following result.
\begin{thm}[Curtis, Hedlund \& Lyndon \cite{Hedlund:1969}]
  \dfn{Cellular automata} are exactly the TDDS $(\am,F)$ such that $F$ commutes with the shift $\sigma$.
\end{thm}

More generally, we define a \dfn{partial cellular automata} (PCA) as a TDDS $(\Sigma,F)$ where $\Sigma \subseteq \am$ and $F$ commutes with the shift.
One can note that Hedlund's theorem still applies and that such systems correspond to restrictions of cellular automata to subshifts. Therefore, they are also defined thanks to an anchor $m$, a diameter $d$ and a local function $f:\lang_d(\Sigma)\to A$.

We say that the PCA is \dfn{oneway}\resp{\dfn{oblic}} if its anchor can be taken $m\le0$ or $m\ge d-1$\resp{$m<0$ or $m\ge d$}. A state $0\in A$ is \dfn{quiescent} for a PCA $(\Sigma,F)$ if $0^d\in\lang(\Sigma)$ and $f(0,\ldots,0)=0$. The global rule $F$ of a PCA can be canonically extended to all words by: \[\appl F{\lang(\Sigma)}{\lang(\Sigma)}u{(f(u\sco i{i+d}))_{0\le i<\length u -d-1}~.}\]
 Without loss of generality, we will sometimes assume the neighborhood to be symmetrical, \ie $d=2m+1$ if $\Mset=\Zset$, $m=0$ if $\Mset=\Nset$; in that case $r=d-m-1$ is called the \dfn{radius} of the CA.


A basic example of CA that we will use throughout the paper is $\Min:\deux^\Nset\to\deux^\Nset$, defined by anchor $0$, diameter $2$ and local rule:
\[\appl f{\deux^2}\deux{(a,b)}{a\times b~.}\]
One easy remark is that this CA admits both $0$ and $1$ as quiescent states.

\subsection{Dynamical properties}
We now study the dynamics of the previously introduced systems $(X,F)$, \ie the structure of the \dfn{orbits} $\orb_F(x)=\sett{F^t(x)}{t\in\Nset}$ of the points $x\in X$. We note the \dfn{positive orbit} $\orb^+_F(x)=\sett{F^t(x)}{t\in\Ns}$ of $x\in X$.
In the case of a TDDS $F$, we can depict such an orbit $\orb_F(x)$ in a space-time diagram which consists in piling up the successive iterates $x,F(x),F^2(x)\ldots$ (see Figure \ref{fig:st-diag}).

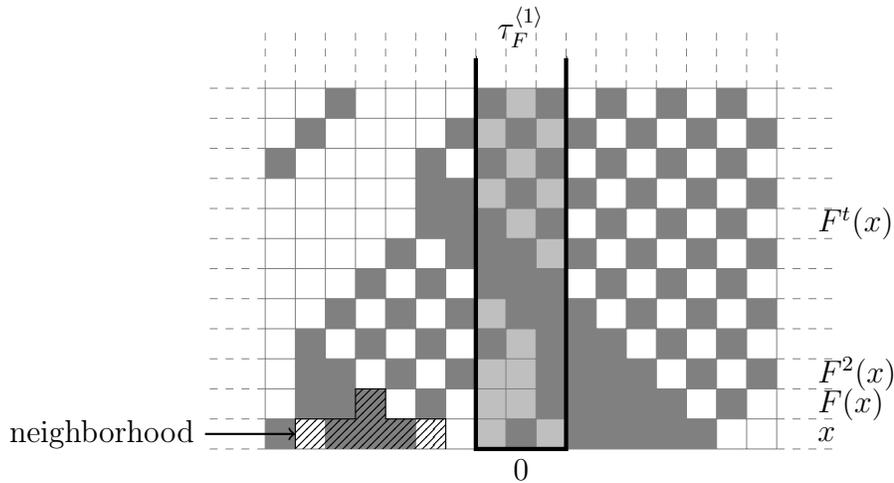
\begin{figure}[htp]\centering
  \begin{tikzpicture}[scale=.4]
\fill[gray] (0,0) rectangle +(1,1)
            (2,0) rectangle +(1,1)
            (3,0) rectangle +(1,1)
            (4,0) rectangle +(1,1)
            (8,0) rectangle +(1,1)
            (10,0) rectangle +(1,1)
            (11,0) rectangle +(1,1)
            (12,0) rectangle +(1,1)
            (13,0) rectangle +(1,1)
            (14,0) rectangle +(1,1)
            (1,1) rectangle +(1,1)
            (2,1) rectangle +(1,1)
            (3,1) rectangle +(1,1)
            (5,1) rectangle +(1,1)
            (9,1) rectangle +(1,1)
            (10,1) rectangle +(1,1)
            (11,1) rectangle +(1,1)
            (12,1) rectangle +(1,1)
            (13,1) rectangle +(1,1)
            (15,1) rectangle +(1,1)
            (1,2) rectangle +(1,1)
            (2,2) rectangle +(1,1)
            (4,2) rectangle +(1,1)
            (6,2) rectangle +(1,1)
            (9,2) rectangle +(1,1)
            (10,2) rectangle +(1,1)
            (11,2) rectangle +(1,1)
            (12,2) rectangle +(1,1)
            (14,2) rectangle +(1,1)
            (16,2) rectangle +(1,1)
            (1,3) rectangle +(1,1)
            (3,3) rectangle +(1,1)
            (5,3) rectangle +(1,1)
            (7,3) rectangle +(1,1)
            (9,3) rectangle +(1,1)
            (10,3) rectangle +(1,1)
            (11,3) rectangle +(1,1)
            (13,3) rectangle +(1,1)
            (15,3) rectangle +(1,1)
            (2,4) rectangle +(1,1)
            (4,4) rectangle +(1,1)
            (6,4) rectangle +(1,1)
            (8,4) rectangle +(1,1)
            (9,4) rectangle +(1,1)
            (10,4) rectangle +(1,1)
            (12,4) rectangle +(1,1)
            (14,4) rectangle +(1,1)
            (16,4) rectangle +(1,1)
            (3,5) rectangle +(1,1)
            (5,5) rectangle +(1,1)
            (7,5) rectangle +(1,1)
            (8,5) rectangle +(1,1)
            (9,5) rectangle +(1,1)
            (11,5) rectangle +(1,1)
            (13,5) rectangle +(1,1)
            (15,5) rectangle +(1,1)
            (4,6) rectangle +(1,1)
            (6,6) rectangle +(1,1)
            (7,6) rectangle +(1,1)
            (8,6) rectangle +(1,1)
            (10,6) rectangle +(1,1)
            (12,6) rectangle +(1,1)
            (14,6) rectangle +(1,1)
            (16,6) rectangle +(1,1)
            (5,7) rectangle +(1,1)
            (6,7) rectangle +(1,1)
            (7,7) rectangle +(1,1)
            (9,7) rectangle +(1,1)
            (11,7) rectangle +(1,1)
            (13,7) rectangle +(1,1)
            (15,7) rectangle +(1,1)
            (5,8) rectangle +(1,1)
            (6,8) rectangle +(1,1)
            (8,8) rectangle +(1,1)
            (10,8) rectangle +(1,1)
            (12,8) rectangle +(1,1)
            (14,8) rectangle +(1,1)
            (16,8) rectangle +(1,1)
            (0,9) rectangle +(1,1)
            (5,9) rectangle +(1,1)
            (7,9) rectangle +(1,1)
            (9,9) rectangle +(1,1)
            (11,9) rectangle +(1,1)
            (13,9) rectangle +(1,1)
            (15,9) rectangle +(1,1)
            (1,10) rectangle +(1,1)
            (6,10) rectangle +(1,1)
            (8,10) rectangle +(1,1)
            (10,10) rectangle +(1,1)
            (12,10) rectangle +(1,1)
            (14,10) rectangle +(1,1)
            (16,10) rectangle +(1,1)
            (2,11) rectangle +(1,1)
            (7,11) rectangle +(1,1)
            (9,11) rectangle +(1,1)
            (11,11) rectangle +(1,1)
            (13,11) rectangle +(1,1)
            (15,11) rectangle +(1,1);
\foreach \i in {0,1,...,12}
  \draw[help lines,dashed] (0,\i) -- +(-2,0)
                (17,\i) -- +(2,0);
 
\draw[help lines] (0,0) grid +(17,12);

\foreach \i in {0,1,...,17}
  \draw[help lines,dashed] (\i,12) -- +(0,2);

      \draw[draw=none,fill=gray,opacity=.5] (7,0) rectangle (10,12);
      \node[anchor=south] at (8.5,13) {$\tau\iexp1_F$};
      \foreach \i in {7,10}
	\draw[dashed,ultra thick] (\i,13) -- (\i,13);
      \draw[ultra thick] (10,13) -- (10,0) -- (7,0) -- (7,13);
      \draw[pattern=north east lines] (1,0) -- (1,1) -- (3,1) -- (3,2) -- (4,2) -- (4,1) -- (6,1) -- (6,0);
      \draw[thick,->] (-2,.5) -- (1,.5);
      \node[anchor=east] at (-2,.5) {neighborhood};
      \node[anchor=west] at (18,.5) {$x$};
      \node[anchor=west] at (18,1.5) {$F(x)$};
      \node[anchor=west] at (18,2.5) {$F^2(x)$};
      \node[anchor=west] at (18,7.5) {$F^t(x)$};
      \node[anchor=north] at (8.5,0) {$0$};
  \end{tikzpicture}
  \caption{Space-time diagram and trace over segment $\cc{-1}1$}\label{fig:st-diag}
\end{figure}

\paragraph{Traces}

In subspaces of $\am$, the continuity implies some concept of \xpr{locality}; we can study what happens to some portion of the configuration. This notion is called \dfn{trace} and is depicted in Figure \ref{fig:st-diag}. It can be formally defined as follows:
\begin{defi}
  The \dfn{trace application} of some TDDS $(\Sigma\subset\am,F)$ in cells $\co ik$, where $i,k\in\Mset$ and $i<k$, is: \[\appl{T_F^{\co ik}}\Sigma{\lang_{k-i}(\Sigma)^\Nset}x{(F^t(x)\sco ik)_{t\in\Nset}~.}\]
\end{defi}

The image $\tau_F^{\co ik}=T_F^{\co ik}(\Sigma)$ is a onesided subshift over alphabet $\lang_{k-i}(\Sigma)$ and $T_F^{\co ik}$ is a factor map of $(\Sigma,F)$ onto $(\tau_F^{\co ik},\sigma)$. In TDDS, finer and finer traces can be used to approach the global system, like observations made with some error represented by the partition. We will note $\tau_F=\tau_F^0$ the central trace.

The traces of some PCA $(\Sigma,F)$ have a very specific property: for any $i,k,h\in\Mset$, $\tau_F^{\co ik}=\tau_F^{\co{i+h}{k+h}}$ thanks to invariance by shift. In particular, we get the following basic property.
\begin{prop}\label{p:concat}
If $(\Sigma,F)$ is a PCA of diameter $d\in\Nset$ and anchor $m\in\sco0d$ on some $(d-1)$-SFT, $q\in\Nset\sqcup\{\infty\}$, $i\in\Mset$ and $x,y\in\Sigma$ two configurations such that $T^{\co i{i+d-1}}_F(x)\sco0q=T^{\co i{i+d-1}}_F(y)\sco0q$. Then for any generation $t\in\co0q$, $F^t(x\conc_iy)=F^t(x)\conc_iF^t(y)$.
\end{prop}
Note that the conditions over the anchor and the order of the SFT are not so restrictive, since we can always enlarge the diameter.

\begin{proof}
We can see by recurrence on $t<q$ that the neighborhood $F^t(x\conc_iy)\sco{k-m}{k-m+d}$ of each cell $k\in\Mset$ corresponds to the neighborhood $F^t(x)\sco{k-m}{k-m+d}$ if $k<m$, $F^t(y)\sco{k-m}{k-m+d}$ otherwise; therefore the application of the local rule remains unchanged.
\end{proof}

\paragraph{Nilpotency, preperiodicity}

We are first interested in the DDS where the dynamics of every point are ultimately very simple (either stable or periodic). Let $(X,F)$ a DDS, $z\in X$ a point. A point $x \in X$ is said \dfn{$z$-nilpotent} if there exists a generation $q\in\Nset$ such that for any $t\ge q$, $F^t(x)=z$. It is said \dfn{$(p,q)$-preperiodic} if $p,q \in \Nset$ such that $F^{p+q}(x)=F^q(x)$. The system is said to be \dfn{weakly $z$-nilpotent}\resp{\dfn{weakly preperiodic}}  if all of its points are $z$-nilpotent\resp{preperiodic}. These definitions allow stronger versions when conditions are uniformized on every points, as follows.
\begin{defi}
  A DDS $(X,F)$ is said \dfn{$z$-nilpotent}, for $z \in X$, if there exists a generation $q\in\Nset$ such that for any $t\ge q$, $F^t(X)=\{z\}$. It is said \dfn{$p$-periodic}\resp{ \dfn{$(p,q)$-preperiodic}} if $F^p=\id$\resp{$F^{p+q}=F^q$}.
\end{defi}
The value $q$ is called the \dfn{preperiod} and $p$ the \dfn{ultimate period}.
We will sometimes speak of periodic, preperiodic or $p$-preperiodic DDS.

The finite DDS are exactly the preperiodic subshifts -- up to conjugacy.

If $(\Sigma\subset\am,F)$ is a PCA, $0\in A$ and $t\in\Nset$ a generation such that for any configuration $x\in\Sigma$, $F^t(x)_0=0$, then it can be seen (thanks to shift-invariance) that $F$ is \ssl{nilpotent}. Conversely, if $(\am,F)$ is a CA which is not $0$-nilpotent for some $0\in A$, then for any $t\in\Nset$, $F^{-t}([A\setminus\{0\}])$ is nonempty and open; in particular, it contains some \ssl{$0$-finite} configuration. Moreover, looking at the dynamics of the finite subsystem of uniform configurations, we can see that for any CA $(\am,F)$, there is a generation $p\in\co0{\card A}$ and a state $0\in A$ which is quiescent for the CA $F^p$, in such a way that the set of \ssl{$0$-finite} configurations is \ssl{$F^p$-invariant}. Summing up the two previous points, we get that for any generation $t\in\Nset$, $F^t(\am)$ contains some $0$-finite nonuniform configuration $z=\pinf0[u]\uinf0$, with $u\in A^*$ and $z_0\ne0$. Last remark but not least, both nilpotency and preperiodicity are preserved under simulation.

We can actually prove that the classes of weakly nilpotent or periodic CA collapse to their strong counterpart.
\begin{prop}
 Any \ssl{weakly preperiodic}\resp{\ssl{weakly nilpotent}} PCA $(\Sigma,F)$ over some transitive subshift $\Sigma$ is \ssl{preperiodic}\resp{\ssl{nilpotent}}.
\end{prop}
\begin{proof}
Our hypothesis consists in decomposing the compact set $\Sigma$ of nonempty interior into the union of subshifts $\bigcup_{q\in\Nset}\bigcup_{p\in\Ns}F^{-q}(\set x\am{F^p(x)=x})$\resp{$\bigcup_{q\in\Nset}F^{-q}(\{z\})$}. By Baire's theorem, one of them has nonempty interior, and thus contains a configuration $x$ which is transitive for $\sigma$. As a subsystem, it shall contain also $\cl{\orb_\sigma(x)}=\Sigma$.
\end{proof}

\paragraph{Transitivity, recurrence, nonwanderingness}
Up to now, we have looked at properties regarding each orbit independently. Let us now take benefit of the topology to study how distinct orbits from a single open set behave.
For a DDS $(X,F)$, a point $x\in X$ is said \dfn{transitive} if its \ssl{positive orbit} is \ssl{dense}: $\cl{\orb^+_F(x)}=X$. It is said \dfn{recurrent} if for any neighborhood $U$ of $x$, there is some generation $t>0$ such that $F^t(x)\in U$. It is \dfn{nonwandering} if for any neighborhoods $U,V$ of $x$, there is some point $y\in U$ and some generation $t>0$ such that $F^t(y)\in V$. Those definitions can be extended to dynamical systems.
A DDS $(X,F)$ is \dfn{transitive} if it admits a residual subset of transitive points and \dfn{nonwandering} if all of its points are nonwandering.

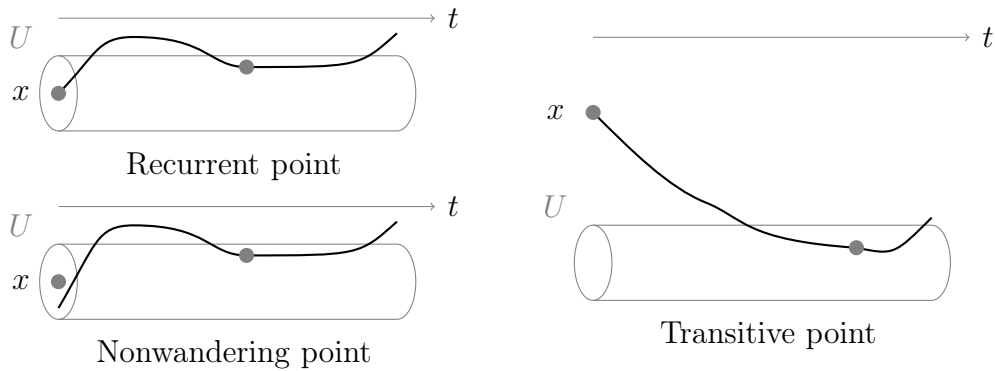
\begin{figure}[!htp]
  \centering
  \begin{tabular}{c c}
\begin{tabular}{c}
    \begin{tikzpicture}
      \draw[help lines,->] (0,1) -- (5,1);
      \node at (5.25,1) {$t$};
      \node at (-.5,0) {$x$};
      \node[gray] at (-.5,.75) {$U$};
      \draw[gray] (0,0) ellipse (.25cm and .5cm);
      \draw[gray] (0,.5) -- (4.5,.5);
      \draw[gray] (0,-.5) -- (4.5,-.5);
      \draw[gray] (4.5,-.5) arc (-90:90:.25cm and .5cm);
      \draw[thick] (0,0)  .. controls (.5,.5) and (0.5,.75) .. (1,.75) .. controls (2,.75) and (2,.35) .. (2.5,.35).. controls (4,.35) .. (4.5,.8);
      \fill[gray] (0,0) circle (.1cm); 
      \fill[gray] (2.5,.35) circle (.1cm); 
    \end{tikzpicture}
\\
Recurrent point
\\
    \begin{tikzpicture}
      \draw[help lines,->] (0,1) -- (5,1);
      \node at (5.25,1) {$t$};
      \node at (-.5,0) {$x$};
      \node[gray] at (-.5,.75) {$U$};
      \draw[gray] (0,0) ellipse (.25cm and .5cm);
      \draw[gray] (0,.5) -- (4.5,.5);
      \draw[gray] (0,-.5) -- (4.5,-.5);
      \draw[gray] (4.5,-.5) arc (-90:90:.25cm and .5cm);
      \draw[thick] (0,-.35)  .. controls (.5,.5) and (0.5,.75) .. (1,.75) .. controls (2,.75) and (2,.35) .. (2.5,.35).. controls (4,.35) .. (4.5,.8);
      \fill[gray] (0,0) circle (.1cm); 
      \fill[gray] (2.5,.35) circle (.1cm); 
    \end{tikzpicture}
\\ Nonwandering point
 \end{tabular}
 &
\begin{tabular}{c}
    \begin{tikzpicture}
      \draw[help lines,->] (0,1) -- (5,1);
      \node at (5.25,1) {$t$};
      \node at (-.5,0) {$x$};
      \node[gray] at (-.5,-1.25) {$U$};
      \draw[gray] (0,-2) ellipse (.25cm and .5cm);
      \draw[gray] (0,-1.5) -- (4.5,-1.5);
      \draw[gray] (0,-2.5) -- (4.5,-2.5);
      \draw[gray] (4.5,-2.5) arc (-90:90:.25cm and .5cm);
      \draw[thick] (0,0) .. controls (.5,-.5) and (1,-1) .. (1.5,-1.2) .. controls (2,-1.4) and (2,-1.7) .. (3.5,-1.8) .. controls (4,-1.9) .. (4.5,-1.4);
      \fill[gray] (0,0) circle (.1cm); 
      \fill[gray] (3.5,-1.8) circle (.1cm); 
    \end{tikzpicture}
\\
   Transitive point\end{tabular} 
  \end{tabular}
  \caption{Recurrence and nonwanderingness}
  \label{fig:rec}
\end{figure}

By compactness, we can have the following equivalent characterization: a DDS $(X,F)$ is \dfn{transitive} if and only if for any nonempty open sets $U,V\subset X$, there is some point $x\in U$ and some generation $t>0$ such that $F^t(x)\in V$. It is \dfn{nonwandering} if any nonempty open set $U\subset X$ contains some point $x\in U$ and some generation $t>0$ such that $F^t(x)\in U$. This is equivalent to having only nonwandering points and to having a residual set of recurrent points, by the following proposition (see \cite{surj}).
\begin{prop}\label{p:recautotr}
Any nonwandering DDS has \ssl{residual} set of \ssl{recurrent} points.
\end{prop}
\begin{proof}
 If $(X,F)$ is a DDS, the set of its recurrent points can be written $\mathcal R=\bigcap_{n\in\Nset}\mathcal R_n$, where $\mathcal R_n=\set xX{\exists t>0,\dist(F^t(x),x)<1/n}$ for all $n\in\Nset$. Note that $\mathcal R_n$ is open. Moreover, if $x\in X$ is nonwandering and $\varepsilon=1/2n$, then by definition there exists some neighbor point $y\in\ball\varepsilon x$ and some generation $t>0$ such that $F^t(y)\in\ball\varepsilon x$; in particular $y\in R_n$. We have proved that each $R_n$ is dense, which gives the result by Baire's theorem.
\end{proof}

\paragraph{Equicontinuity, sensitivity}
Now, let us take some alternative observation method and study how large changes can appear when introducing a small change in the configuration.

Let $(X,F)$ a DDS, $\varepsilon\in\Rset_+\setminus\{0\}$. A point $x\in X$ is said \dfn{$\varepsilon$-unstable} if for any radius $\delta>0$, there is a point $y\in\ball\delta x$ and a generation $t\in\Nset$ for which $\dist(F^t(x),F^t(y))>\varepsilon$. Otherwise the point is said \dfn{$\varepsilon$-stable} (see Figure \ref{fig:stab}). A point which is $\varepsilon$-stable for any $\varepsilon>0$ is said \dfn{equicontinuous}.

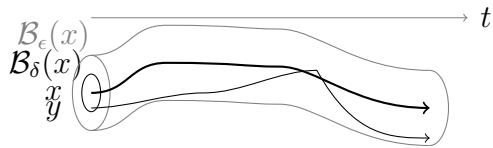
\begin{figure}[!htp]
  \centering
  \begin{tikzpicture}
      \draw[help lines,->] (0,1) -- (5,1);
      \node at (5.25,1) {$t$};
    \node at (-.5,0) {$x$};
    \node at (-.5,-.25) {$y$};
    \node[gray] at (-.5,.75) {$\ball\epsilon x$};
    \node at (-.6,.375) {$\ball\delta x$};
    \draw[gray] (0,0) ellipse (.25cm and .5cm);
    \begin{scope}[yshift=.5cm]
    \draw[gray] (0,0)  .. controls (.5,0) and (0.5,.4) .. (1,.4) .. controls (2,.4) and (2,.35) .. (2.5,.35) .. controls (3,.35) and (3.5,-.2) .. (4.5,-.2); 
    \end{scope}
   \begin{scope}[yshift=-.5cm]
    \draw[gray] (0,0)  .. controls (.5,0) and (0.5,.4) .. (1,.4) .. controls (2,.4) and (2,.35) .. (2.5,.35) .. controls (3,.35) and (3.5,-.2) .. (4.5,-.2); 
    \end{scope}
    \draw[gray] (4.5,-.7) arc (-90:90:.25cm and .5cm);
    \draw[->,thick] (0,0)  .. controls (.5,0) and (0.5,.4) .. (1,.4) .. controls (2,.4) and (2,.35) .. (2.5,.35) .. controls (3,.35) and (3.5,-.2) .. (4.5,-.2);
    \draw (0,0) ellipse (.125cm and .25cm);
    \draw[->] (0,-.2)  .. controls (.5,-.2) and (1,0) .. (1.5,0) .. controls (2,0) and (2.5,.3) .. (3,.3) .. controls (3.5,-.6) and (4,-.6) .. (4.5,-.6);

\end{tikzpicture}
 \caption{An $\varepsilon$-stable point.}
  \label{fig:stab}
\end{figure}

\begin{defi}
A DDS $F$ is said \dfn{$\varepsilon$-sensitive} if all of its points are \ssl{$\varepsilon$-unstable}, with $\varepsilon>0$. It is said \dfn{almost equicontinuous} if its set of equicontinuous points is a residual. It is \dfn{equicontinuous} if for any radius $\varepsilon>0$, there exists a radius $\delta>0$ such that for all points $x,y\in X$ with $\dist(x,y)<\delta$ and all generation $t\in\Nset$ we have $\dist(F^t(x),F^t(y))<\varepsilon$.
\end{defi}
Due to the compactness of the underlying space, it is possible to invert the two quantifiers in the definition of equicontinuity and to achieve the following characterization: a DDS $F$ is equicontinuous if and only if all of its points are.

A first example of equicontinuous systems is the preperiodic ones: their behavior only depends on the beginning of their orbit.
\begin{prop}\label{p:prepequi}
 Any \ssl{preperiodic} DDS is \ssl{equicontinuous}.
\end{prop}
\begin{proof}
 Let $(X,F)$ be a $(p,q)$-preperiodic DDS with $q\in\Nset$, $p\in\Ns$ and $\varepsilon>0$. Then each iterate $F^t$, for $t\in\Nset$, is uniformly continuous, \ie there exists $\delta_t>0$ such that for all points $x,y\in X$ with $\dist(x,y)<\delta_t$, we have $\dist(F^t(x),F^t(y))<\varepsilon$. Since $F^t=F^{q+(t-q)\bmod p}$, we can define $\delta=\min_{0\le j<p+q}\delta_j$, in such a way that for all generation $t\in\Nset$ and all points $x,y\in X$ with $\dist(x,y)<\delta$, we have $\dist(F^t(x),F^t(y))<\varepsilon$.
\end{proof}

In the specific case of a TDDS $(\Sigma,F)$, the definition of an equicontinuous point $x\in\Sigma$ can be reformalized in terms of the trace as follows: for any $k\in\Nset$, there exists $l\in\Nset$ such that $T\iexp k_F([x\isub l])$ is a singleton. We can further characterize the notion of equicontinuity as follows.
\begin{prop}\label{p:eqtr}
 A TDDS $(\Sigma,F)$ is \ssl{equicontinuous} if and only if all of its \ssl{traces} are \ssl{finite}.
\end{prop}
\begin{proof}
Let $F$ be an equicontinuous TDDS and $k\in\Nset$. There exists a radius $l\in\Nset$ such that for any $u\in A\iexp l$, $T_F\iexp k([u])$ is a singleton. Consequently, $\tau_F\iexp k=\bigcup_{u\in A\iexp l}T_F\iexp k([u])\le\card{A\iexp l}$.
\\
Conversely, if $\tau_F\iexp k$ is finite, then it is $(p,q)$-preperiodic, for some $p\in\Ns$ and $q\in\Nset$. Any point $x\in\Sigma$ is $\varepsilon$-stable, since any point $y$ of the neighborhood $\bigcap_{t<p+q}F^{-t}(\ball\varepsilon{F^t(x)})$ satisfies $\forall t\in\Nset,\dist(F^t(x),F^t(y))<\varepsilon$.
\end{proof}

In the very particular case of PCA, homogeneity allows to note that if all the cells have preperiodic traces with the same period and preperiod, then the whole configuration is preperiodic. Thus, we can state that a PCA $F$ is preperiodic if and only if each of its traces $\tau\iexp k_F$, with $k\in\Nset$, is finite. The trace of width $1$ being the projection of all other traces, the period and preperiod can be uniformized, which leads to a simple generalization of a classical result over CA or PCA on very particular subshifts \cite{classif,newtopifip}.
\begin{cor}\label{c:equiprep}
Any PCA is \ssl{equicontinuous} if and only if it is \ssl{preperiodic}.
\end{cor}

Concerning sensitivity, it is not transmitted to any trace, but it is to sufficiently fine traces as shown by the following proposition.
\begin{prop}\label{p:senstr}
 Let $(\Sigma,F)$ be an \ssl{$\varepsilon$-sensitive} TDDS with $\varepsilon\ge2^{-k}$. Then $\tau_F\iexp k$ is a \ssl{sensitive} subshift.
\end{prop}
\begin{proof}
 Let $x\in\Sigma$ and $\delta>0$. By continuity of the trace application, there exists $\delta'>0$ such that for any configuration $y\in\ball{\delta'}x$, we have $\dist(T_F\iexp k(x),T_F\iexp k(y))<\delta$. The sensitivity of $F$ gives a configuration $y\in\ball{\delta'}x$ and a generation $t\in\Nset$ such that $\dist(F^t(x),F^t(y))>\varepsilon$, \ie $F^t(x)\isub k\ne F^t(y)\isub k$. As a result, $T_F\iexp k(x)_t\ne T_F\iexp k(y)_t$, \ie $\dist(\sigma^tT_F\iexp k(x),\sigma^tT_F\iexp k(y))=1$, with $\dist(T_F\iexp k(x),T_F\iexp k(y))<\delta$.
\end{proof}

In the space $\am$, stability can be linked with blocking words, defined as follows: a word $w\in A^*$ is \dfn{$k$-blocking} for the TDDS $(\Sigma,F)$ if there exists $i\in\Nset$ such that $\forall x,y\in[w]_{-i},\forall j\in\Nset,F^j(x)\sco0k=F^j(y)\sco0k$. 

The reader can easily note that a word is \ssl{$k$-blocking} if one of its patterns is, and that any \ssl{$k$-blocking} word is \ssl{$i$-blocking} for all $i\le k$.
Moreover, if $(\Sigma,F)$ is a TDDS and $k\in\Nset$, then a configuration $x\in\Sigma$ is \ssl{$2^{-k}$-stable} if and only if $x\isub l$ is \ssl{$k$-blocking} for some $l\in\Nset$, which brings the following remark.
\start{rem}\label{r:bloqequi}
 \item A TDDS is \ssl{$2^{-k}$-}{sensitive} if and only if it does not admit any \ssl{$k$-blocking} word.
 \item A configuration is \ssl{equicontinuous} if and only if it admits \ssl{$k$-blocking} central patterns for any $k\in\Nset$.
\finish{rem}

But blocking words are especially interesting regarding CA, since a particular width is enough to block all widths. Intuitively, these blocking words will disconnect the underlying space into two different components, preventing future information transfers between them.

If $(\Sigma,F)$ is a PCA of radius $r$, $w$ a $r$-blocking word, $i\in\Nset$ as in the definition, and $x\in[w]_{-i}$, then for any configuration $y\in\Sigma$ with $y\sci{-i}=x\sci{-i}$\resp{$y\sio{\length w-i}=x\sio{\length w-i}$} and any generation $t\in\Nset$, we have $F^t(y)\sci0=F^t(x)\sci0$\resp{$F^t(y)\sio r=F^t(x)\sio r$}.

For instance, in the $\Min$ CA, the word $0$ is $1$-blocking, since $\forall x\in[0],\forall t\in\Nset,F^t(x)_0=0$; its radius being $1$, any space-time diagram containing $0$ can be separated into two parts evolving independently.

\begin{rem}\label{r:bloqbloq}
Let $(\Sigma,F)$ be a PCA of radius $r\in\Nset$, $i,j\in\Mset$, $k,l\ge r$, and $u,v$ two words which are respectively \ssl{$k$-blocking} and \ssl{$l$-blocking}, $i$ and $i'$ the corresponding indices in the words (from the definition).
If the concatenation $uv$ is in the language $\lang(\Sigma)$, then it is \ssl{$\length u-i+j+l$-blocking}.
\end{rem}
This last fact implies the following proposition.
\begin{prop}[K{\r{u}}rka \cite{classif}]
 Let $(\Sigma,F)$ a PCA of radius $r$.
 Then $F$ is \ssl{equicontinuous} if and only if there exists $k\in\Nset$ such that all the words of $A^k$ are \ssl{$r$-blocking}.
\end{prop}

Another consequence of Remark \ref{r:bloqbloq} is that we can insert any word between two concatenated words and obtain arbitrarily wide blocking words in any cylinder. We obtain the following theorem, equivalent to a result in \cite{classif}.
\begin{thm}\label{t:quasieq}
 Let $(\Sigma,F)$ a PCA of radius $r$ on some transitive subshift. The following statements are equivalent:
 \begin{enumerate}
  \item $F$ is almost equicontinuous;
  \item $F$ is not $2^r$-sensitive;
  \item $F$ admits some $r$-blocking word.
 \end{enumerate}
\end{thm}
\start{proof}
 \item Assume that $u\in A^*$ is a $k$-blocking word for $F$, with $k\ge r$, and let us show that the set of equicontinuous configurations is a residual.
By transitivity of $\Sigma$, the open set $U_l=\bigcup_{j>l}[u]_j$ is dense. Thanks to the Remarks \ref{r:bloqbloq} and \ref{r:bloqequi}, the configurations of the intersection $\bigcap_{l\in\Nset}U_l$ are equicontinuous.
 \item The other implications directly come from the definitions.
\finish{proof}

\section{Limit set}\label{s:ls}

The previous section has presented many different possible ways to study dynamics of DDS and some first results about them. However, these notions can be very sensitive to the transient time (if we modify the initial evolution during a short time). To overcome this problem and characterize the core behavior of the DDS, one idea is to consider only the points that can appear arbitrarily late inside the DDS. This corresponds to the limit set. Formally, if $(X,F)$ is a DDS and $X'\subset X$, we note ${\Omega_F(X')}=\bigcap_{j\in\Nset}\cl{\orb_F(F^j(X'))}$. In case of an invariant set (such as $X$), the definition gets simpler.

\begin{defi}
  Let $(X,F)$ a DDS. The \dfn{limit set} of a DDS $(X,F)$ is $\Omega_F=\Omega_F(X)=\bigcap_{j\in\Nset}F^j(X)$. 
\end{defi}

 The limit set of the $\Min$ CA is the set of the configurations where all the $1$s are connected:
 \[\Omega_{\Min}=\set x{\deux^\Zset}{\forall i\in\Zset,x_i=0\impl\forall j<i,x_j=0\oub\forall j>i,x_j=0}~.\]

\subsection{Limit set of DDS}

One first easy remark is that the limit set is always a closed nonempty set, as a decreasing intersection of nonempty closed subsets. 

The limit set corresponds to the largest \ssl{surjective subsystem}; in particular $\Omega_F=X$ if and only $F$ is \ssl{onto}. Let us now define an \dfn{attractor} as a set that attracts neighboring points or, formally, a nonempty closed $F$-invariant subset $Y$ of $X$ such that for any $\varepsilon>0$, there exists $\delta>0$ such that for any point $x\in X$ with $\dist(x,Y)<\delta$, we have $\lim_{j\to\infty}\dist(F^j(x),Y)=0$ and for any generation $j\in\Nset$, $\dist(F^j(x),Y)<\varepsilon$. The limit set is then the \ssl{maximal attractor}. Using this characterization and the compactness of the underlying space, it can be shown that any neighborhood is reached in a finite time, \ie $\max_{x\in X}\dist(F^j(x),\Omega_F)\tnd[j]0$. In the case where the limit set is reached in a finite uniform time, \ie there exists a generation $j\in\Nset$ such that $F^j(X)=\Omega_F$, we say that the DDS $(X,F)$ is \dfn{stable}.

If $(X_1,F)$ and $(X_2,F)$ are two subsystems of $(X,F)$ such that $X_1\cup X_2=X$, then $\Omega_F=\Omega_F(X_1)\cup\Omega_F(X_2)$.

We can build, using the limit set, similar notions to those we already introduced.
\begin{defi}\label{d:enslim}
  Let $(X,F)$ be a DDS. It is said \dfn{$z$-limit-nilpotent} for some $z\in X$, if $\Omega_F$ is the singleton $\{z\}$. It is said \dfn{$p$-limit-periodic}, with $p\in\Ns$, if $F\restr\Omega$ is \ssl{$p$-periodic} (\ie, $F\restr\Omega^p=\textrm{Id}$).
\end{defi}   

 A first remark is that any $z$-nilpotent DDS is $z$-limit nilpotent, and similarly preperiodic DDS are limit-periodic, but the converse is not always true. Nevertheless these new notions represent some highly stable behavior, as seen in the following proposition.
\begin{prop}\label{p:lnilpeq}
 Any \ssl{limit-nilpotent} DDS is \ssl{equicontinuous}.
\end{prop}
\begin{proof}
  Let $(X,F)$ be a $z$-limit-nilpotent DDS for some $z\in X$, and $\varepsilon>0$. There exists a generation $J\in\Nset$ such that $\forall j\ge J,\max_{x\in X}\dist(F^j(x),z)<\frac\varepsilon2$. Hence, for any point $x\in X$ and any point $y$ of the open set $\bigcap_{0\le j<J}F^{-j}(\ball\varepsilon{F^j(x)})$, we have by construction $\dist(F^j(x),F^j(y))<\varepsilon$ for $j<J$ and $\dist(F^j(x),F^j(y))<\frac\varepsilon2+\frac\varepsilon2$. It results that $x$ is $\varepsilon$-stable.
\end{proof}

Moreover, we can remark that a system is $p$-limit-periodic if and only if its limit set is the set of its $p$-periodic points.

The limit set of a subsystem is included in that of the whole system. Hence any subsystem of a limit-nilpotent\resp{limit-periodic} is limit-nilpotent\resp{limit-periodic}. Moreover, the limit set is preserved under iteration and factor map. Thus, we can see that if $\Phi$ is a simulation by a DDS $(X,F)$ of another $(Y,G)$ then $\Omega_G(Y) \subseteq \Phi(\Omega_F(X))$, and we have the following.
\begin{prop}\label{p:limsim}
Let $\Phi:X\to Y$ be a \ssl{complete simulation} by a DDS $(X,F)$ of another $(Y,G)$.
Then $\Phi(\Omega_F(X))=\Omega_G(Y)$.
\end{prop}
\begin{proof}
Let $\Phi$ be a factor map.
For any $j\in\Nset$, $\Phi F^j(X)=G^j(Y)$, hence $\Phi(\bigcap_{j\in\Nset}F^j(X))=\bigcap_{j\in\Nset}G^j(Y)$ since it is a decreasing intersection; hence $\Phi(\Omega_F)=\Omega_G$.
Moreover, decreasingness of the sequence clearly gives $\Omega_{F^k}=\Omega_F$ for any $k\in\Ns$.
\end{proof}

\subsection{Limit set of subshifts}\label{d:limsss}

Here, we will study what happens when we are in the case of TDDS or, in particular, of onesided subshifts (obviously, the limit system of a twosided subshift is the whole subshift). For a onesided subshift $\Sigma$, we will more conveniently note ${\Omega_\Sigma}=\Omega_\sigma(\Sigma)$.

 For the label system of a graph, the limit set can be read by removing the inaccessible vertices (until none remain). More formally, the \ssl{limit set} of the \ssl{label system} of a given graph is the subgraph composed of vertices that are \ssl{accessible} by an infinite path. In this case, we have the property that \ssl{sofic} subshifts are \ssl{stable}. Actually, as soon as the limit set is an SFT, the next proposition and corollary show that it is reached in finite time.
\begin{prop}\label{p:limsft} 
Any subshift having a \ssl{limit set} of \ssl{finite type} is \ssl{stable}. 
\end{prop}
\begin{proof} Let $\Sigma$ be a subshift such that $\Omega_\Sigma$ is an SFT of order $k\in\Nset$. $[\lang_k(\Omega_\Sigma)]$ is a neighborhood of $\Omega_\Sigma$, so it is reached in finite time (from previous remarks): there exists a generation $t\in\Nset$ for which $\sigma^t(\Sigma)$ is included in $[\lang_k(\Omega_\Sigma)]$.
Being a subshift, $\sigma^t(\Sigma)$ must also be included in $\bigcap_{j\in\Nset}\sigma^j([\lang_k(\Omega_\Sigma)])$, which is exactly $\Omega_\Sigma$ since it is an SFT of order $k$.
\end{proof} 

\begin{cor}\label{c:limsft}
 A subshift is \ssl{finite}\resp{of \ssl{finite type}} if and only if its \ssl{limit set} is.
\end{cor}
\begin{proof}
  Let $\Sigma\subset\am$ be a subshift such that $\Omega_\Sigma$ is an SFT of order $k\in\Ns$. By proposition \ref{p:limsft}, there exists a generation $j\in\Nset$ such that $\sigma^j(\Sigma)=\Omega_\Sigma$. It is then immediate that $\Sigma$ is a $(k+j)$-SFT and that $\card\Sigma\le\card{A^j}\card{\Omega_\Sigma}$. The converse is immediate by the previous remark on the limit set of a label system.
\end{proof}

In other words, the \ssl{limit-periodic}\resp{\ssl{limit-nilpotent}} subshifts are exactly the \ssl{preperiodic}\resp{\ssl{nilpotent}} subshifts. The argument of the previous proof cannot be adapted to sofic subshifts, as shown by the following counter-example: the subshift $\sett{0^k1^l\uinf0}{k\le l}+0^*\uinf1$ is not sofic, even though its limit set $1^*\uinf0+0^*\uinf1$ is sofic. 

\subsection{Limit set of cellular automata}

In the case of (partial) cellular automata or TDDS, the particular structure allows stronger results. Nevertheless, it is not completely understood, as suggests the attempt to characterize the possible limit sets of CA in \cite{soflim}, or more generally the possible subshift attractors in \cite{sssattr}, also linked to \cite{classifbass}. Generally, it is known that the limit sets of CA can be rather complex \cite{rice,lrice}.

First note that the \ssl{limit set} of a PCA $(\Sigma,F)$ is a \ssl{subshift}, as an intersection of subshifts. Moreover, its language is the limit $\lang(\Omega_F)=\bigcap_{j\in\Nset}\lang(F^j(\Sigma))$ of the languages of the successive image subshifts. In the case of TDDS, we can generalize Proposition \ref{p:lnilpeq} to obtain a strong condition of stability via the limit set.
\begin{prop}\label{p:limper}
 Any \ssl{limit-periodic} TDDS is \ssl{equicontinuous}.
\end{prop}
\begin{proof}
 Let $(\Sigma,F)$ be a limit-periodic TDDS. Then Proposition \ref{p:limsim} gives that its traces are all limit-periodic too, and Corollary \ref{c:limsft} that they are preperiodic. Proposition \ref{p:eqtr} allows then to conclude that $F$ is equicontinuous.
\end{proof}

As far as nilpotency is concerned, it is obvious that a TDDS is nilpotent if and only if all of its traces are nilpotent. For a PCA $F$, as all of them share the same projection of width $1$, the characterization is simpler: $F$ is \ssl{nilpotent} if and only if the central trace $\tau_F$ is \ssl{nilpotent}. The case of period $p=1$ gives us a generalization of a well-known characterization of CA nilpotency \cite{Culik:1989}.
\begin{prop}
 Any PCA is \ssl{nilpotent} if and only if it is \ssl{limit-nilpotent}.
\end{prop}

In the case of a full CA, we can prove some restriction on the limit set showing that the nilpotent behavior can be ``isolated'' from other behaviors: if a CA is not nilpotent, its limit set will contains numerous configurations. 
\begin{prop}\label{p:semi0}
 Let $(\am,F)$ a \ssl{non $0$-nilpotent} CA, with $0\in A$ and $\dinf0\in\Omega_F$. Then $\Omega_F$ contains, for any $k\in\Mset$, a \ssl{semifinite} configuration $z\ne\dinf0$ such that $z_i=0$ for all cell $i<k$.
\end{prop}
\begin{proof}
We can consider without loss of generality that the CA is twosided.
By a previous remark, for any generation $j\in\Nset$, $F^j(\az)$ contains some $0$-finite nonuniform configuration $z\in\pinf0[u]\uinf0$, with $u\in A^+\setminus0^+$. Composing with a shift, we obtain $F^j(\am)\cap\pinf0[\compl0]_k\ne\emptyset$ and compactness gives $\Omega_F\cap\pinf0[\compl0]_k\ne\emptyset$.
\end{proof}

As a consequence, we have another characterization of nilpotency.
\begin{cor}
  A CA is \ssl{nilpotent} if and only if it admits some \ssl{isolated uniform} configuration.
\end{cor}
\begin{proof}
  Let $(\am,F)$ a non-nilpotent CA, $0\in A$, $k\in\Nset$; Proposition \ref{p:semi0} gives some nonuniform configuration in $\ball{2^{-k}}x\cap\Omega_F$.
The converse is obvious.
\end{proof}

The previous result allows us to obtain some well-known fact on the cardinality of the limit set of a cellular automaton.
\begin{prop}[{\v{C}}ul\'{\i}k, Pachl \& Yu \cite{Culik:1989}]
The \ssl{limit set} of any CA is either a singleton or infinite.
\end{prop}

An infinite limit set can be countable as the $\Min$ CA, or uncountable, as for surjective CA, in which case, being a subshift, it has a continuous cardinality.
The dichotomy of the previous proposition is no more true for PCA, for instance on finite subshifts.

\section{Asymptotic set}\label{s:as}

If the limit set characterizes the set of points that can appear arbitrarily late during the evolution of the dynamical systems, it may actually contain points which look transient. This is the case of configurations of the form $\pinf0 1\ldots1 \uinf0$ for the $\Min$ CA: we know they will disappear soon. To better emphasize the asymptotic behavior, we study here the set containing all the points for which there exists an evolution of the dynamical system going an infinite number of times close to this point.
\begin{defi}\label{d:ensult}
 Let $(X,F)$ be a DDS. The \dfn{asymptotic set} of a set $X'\subset X$ is the set ${\omega_F(X')}=\bigcup_{x\in X'}\Omega_F(\{x\})$ of \ssl{adhering values} of orbits.
We note ${\omega_F}=\omega_F(X)$.
\end{defi}
 This set was called \emph{ultimate set} in \cite{hdr,nilpeng}, or \emph{accessible set} in \cite{durand}. 
For instance, the asymptotic set of the $\Min$ CA is $\{\dinf0,\dinf1\}$, and is strictly included in its limit set.

\subsection{Asymptotic set of DDS}

Like the limit set, the asymptotic set can be expressed 
by a metric property: it is the smallest subset $Y\subset X$ such that for any $x\in X$, $\dist(F^j(x),Y)\tnd[j]0$. In other words, for any neighborhood $U$ of $\omega_F$ and any point $x\in X$, there exists a generation $J\in\Nset$ such that $\forall j\ge J,F^j(x)\in U$.

 We can immediately see that if $X'\ne\emptyset$, then $\omega_F(X')$ is nonempty and \ssl{$F$-invariant}, but need not be closed (as opposed to the limit set).  The asymptotic set is also always a subset of the limit set: $\omega_F(X')\subset\Omega_F(X')$. One important problem is to understand the dynamics of the orbits which are in the difference of the two sets. 
First note that all the periodic points are contained in the asymptotic set. The following propositions go further.
\begin{prop}\label{p:ulttrans}
  The asymptotic set of a DDS $(X,F)$ contains all of its transitive subsystems.
\end{prop}
\begin{proof} 
Consider a transitive subsystem $(Y\subset X,F)$. Then there is a point $y\in Y$ which is transitive for this subsystem, \ie any point of $Y$ is an adhering value of $\orb_F(y)$.  
\end{proof}
Example \ref{x:minpenche} will show that the inclusion can be strict.

It is known that the set of \emph{uniformly recurrent} points is the union of the \emph{minimal} subsystems (see for instance \cite{kurka} for definitions). Similarly, we can prove the following proposition.
\begin{prop}
 For any recurrent point $x$, the subsystem $\cl{\orb_F(x)}$ is transitive.
\end{prop}
\begin{proof}
 In general, the closure $\cl{\orb_F(x)}$ of the orbit is the union of the closure $\cl{\orb^+_F(x)}$ of the positive orbit and of the singleton $x$. By the property of recurrence, $x\in\cl{\orb^+_F(x)}$. Hence $x$ has a dense positive orbit in this subsystem.
\end{proof}
Nevertheless, the set of transitive subsystems also includes other points: see for instance the case of the full shift, which is transitive, but admits some non-recurrent points.
The two last propositions give that the asymptotic set contains the set of recurrent points; actually it can easily be seen that they are exactly the points which are an adhering value of their own orbit.
 On the other hand, we can show that it is a subset of the set of nonwandering points.
\begin{prop}
 Any point of the \ssl{asymptotic set} of a DDS is \ssl{nonwandering}.
\end{prop}
\begin{proof}
 Let $(X,F)$ a DDS, $\varepsilon>0$ and $x\in\omega_F$, \ie there exists a point $y\in X$ whose orbit admits $x$ as adhering value; in particular, it goes an infinite number of times in the ball $\ball\varepsilon x$. Therefore, there exist some point $y'=F^J(y)\in\ball\varepsilon x$ and some generation $j\in\Ns$ such that $F^j(y')=F^{J+j}(y)\in\ball\varepsilon x$.
\end{proof}

The main interest in these inclusions is that they are \xpr{not far} from each other: from the remark that the set of nonwandering points is closed and from Proposition \ref{p:recautotr}, we deduce the following characterization: a DDS $(X,F)$ is \ssl{nonwandering} if and only if its \ssl{asymptotic set} $\omega_F$ is a \ssl{residual subset} of $X$.

The long-term behavior of the orbits of a system tends to look more and more like the behavior on the asymptotic set, it is therefore relevant to study the case when asymptotic points have a simple evolution, as we have done for the limit set.

\begin{defi}
  A DDS $(X,F)$ is \dfn{asymptotically $z$-nilpotent} if all of its orbits converge towards the same limit $z\in X$, \ie $\omega_F=\{z\}$.  It is \dfn{asymptotically $p$-periodic}, with $p\in\Ns$, if the restricted map $F\restr{\omega_F}$ is $p$-periodic.
\end{defi}

With these definitions, if $F$ is an \ssl{asymptotically $z$-nilpotent} DDS, then $z$ is a fix point of $F$, since $\omega_F$ is $F$-invariant; in particular, $F$ is \ssl{asymptotically $1$-periodic}.  Moreover, for any $\varepsilon>0$, there exists a generation $J\in\Nset$ such that for any point $x\in X$, $\exists j<J,\dist(F^j(x),z)<\varepsilon$.

It is possible to link these behaviors with the previously-defined ones. The first easy point is that \ssl{weakly nilpotent} DDS are \ssl{asymptotically nilpotent}. However, the converse is not true. A simple counter-example is the division by $2$ on interval $[0,1]$. Nevertheless, asymptotically nilpotent DDS cannot be too much unstable, as formalized by the following proposition.

\begin{prop}\label{p:nilbloc}
 No \ssl{asymptotically nilpotent} DDS is \ssl{sensitive}.
\end{prop}
\begin{proof}
  Let $(X,F)$ an asymptotically $z$-nilpotent DDS, with $z\in X$, and $\varepsilon>0$. By definition, the space $X$, of nonempty interior, can be decomposed as a union $\bigcup_{J\in\Nset}\bigcap_{j>J}F^{-j}(\ball{\varepsilon/2}z)$ of closed subsets. By Baire's theorem, there exists a generation $J\in\Nset$ such that the closed subset $\bigcap_{j>J}F^{-j}(\ball{\varepsilon/2}z)$ contains an open subset $U$ of nonempty interior. Let $x\in U$.
 The finite intersection $U\cap\bigcap_{j\le J}F^{-j}(\ball\varepsilon{F^j(x)})$ is then open and contains $x$; consequently, it contains an open ball $\ball\delta x$, with $\delta>0$. For any point $y\in\ball\delta x$ and any generation $j\le J$, we have by construction $\dist(F^j(x),F^j(y))\le\varepsilon$; for any generation $j>J$, we have the triangular inequality $\dist(F^j(x),F^j(y))\le\dist(F^j(x),z)+\dist(z,F^j(y))\le\varepsilon$. As a result, the point $x$ is $\varepsilon$-stable.
\end{proof}

Let us look at how the asymptotic set can be related to the notions of simulation.
\begin{prop}\label{p:ultsim}
Let $\Phi$ a \ssl{complete simulation} by a DDS $(X,F)$ of another $(Y,G)$.
Then $\Phi(\omega_F(X))=\omega_G(Y)$.
\end{prop}
\begin{proof} 
Suppose that $\Phi$ is a factor map. Let $x\in\omega_G$, \ie $x$ is the limit of a subsequence $(G^{k_j}\Phi(y))_{j\in\Nset}$ where $(k_j)_{j\in\Nset}$ is an increasing sequence of integers. Then $(F^{k_j}(y))_{j\in\Nset}$ admits a adhering value $z$, whose image is $\Phi(z)=x$. Hence $x\in\Phi(\omega_F)$. The converse is immediate.

It is now sufficient to show that any DDS $F^k$ has the same asymptotic set than $F$.
First, the decreasingness of the sequence gives $\omega_{F^k}\subset\omega_F$. Then,
let $x\in\omega_F$, \ie $x$ is the limit of some subsequence $(F^{k_j}(y))_{j\in\Nset}$, where $(k_j)_{j\in\Nset}$ is an increasing sequence of integers. By the pigeon-hole principle, there exists some integer $r<k$ such that $J=\set j\Nset{k_j\bmod k=r}$ is infinite. We can see that $(F^{k_j-r}(y)_{j\in J}$ is a subsequence of the orbit of $F^r(y)$ by $F^k$ that admits $x$ as an adhering value. Hence $x\in\omega_{F^k}$.  
\end{proof}
Moreover, the asymptotic set of some subsystem is contained in the asymptotic set of the global system. We even have that, if $(X_1,F)$ and $(X_2,F)$ are two subsystems of $(X,F)$ such that $X_1\cup X_2=X$, then $\omega_F=\omega_F(X_1)\cup\omega_F(X_2)$.
 In particular, like nilpotency and preperiodicity, \ssl{asymptotic nilpotency} and \ssl{asymptotic periodicity} are transmitted by any \ssl{simulation}.

\subsection{Asymptotic set of subshifts}

Let us study how the asymptotic set is constrained in the particular case of subshifts. To lighten the reading, we will note ${\omega_\Sigma}=\omega_\sigma(\Sigma)$.

We can see that the asymptotic set of sofic subshifts can be seen from the graph of their limit sets by removing all the links between strongly connected components. That result can also be restated as follows: the \ssl{asymptotic set} of a sofic subshift is the disjoint union of its maximal \ssl{transitive subsystems}.

In particular, in the onesided case, the asymptotic set is reached by each orbit:
if $\Sigma$ is a onesided sofic subshift and $z\in\Sigma$, then there exists a generation $j\in\Nset$ such that $\sigma^j(z)\in\omega_\Sigma$.

Using regularity of sofic subshifts, it is possible to characterize the notion of asymptotic periodicity: a sofic subshift is \ssl{asymptotically periodic} if and only if it is the \ssl{label system} of some graph in which all the \ssl{strongly connected components} are \ssl{cycles}. In the case of onesided sofic subshifts, the reachability of the asymptotic set from any orbit shows that asymptotic periodicity is equivalent to weak preperiodicity. This can be generalized as follows.
\begin{prop}\label{p:fnilpper2}
Any onesided subshift is \ssl{asymptotically periodic} if and only if it is \ssl{weakly preperiodic}.
\end{prop}
\begin{proof}
 Let $\Sigma\subset\an$ be an asymptotically periodic subshift of period $p\in\Ns$. The open subset $U=\set x\Sigma{x_0=x_p}$ is a neighborhood of $\omega_\Sigma$. For any configuration $x\in\Sigma$, there is a generation $J\in\Nset$ such that for any $j\ge J$, $\sigma^j(x)_0=\sigma^j(x)_p$, \ie $x$ is $(J,p)$-preperiodic. The converse is immediate. 
\end{proof}

We can use the previous proposition to get a generalization of Proposition \ref{p:nilbloc} in that setting.

\begin{cor}\label{c:fnilpper4}
 Any \ssl{asymptotically periodic} onesided subshift is \ssl{almost equicontinuous}.
\end{cor}
\begin{proof} 
  Let $\Sigma$ be an asymptotically periodic onesided subshift of period $p\in\Ns$, and $\varepsilon>0$. By Proposition \ref{p:fnilpper2}, $\Sigma=\bigcup_{j\in\Nset}F^{-j}(\bigcap_{i\in\Nset}\set x\Sigma{x_i=x_{i+p}})$. By Baire's theorem, there is some finite time $J\in\Nset$ and some nonempty open set $U\subset\bigcup_{j<J}F^{-j}(\bigcap_{i\in\Nset}\set x\Sigma{x_i=x_{i+p}})$. If $x\in U$, then the intersection $V= U\cap\bigcap_{0\le j<J+p}\sigma^{-j}(\ball\varepsilon{\sigma^j(x)})$ is still open. For any $y\in V$ and any generation $j\in\Nset$, we have $F^j(y)=F^{J+(j-J\bmod p)}(y)$ and $F^j(x)=F^{J+(j-J\bmod p)}(x)$; by construction, their distance is less than $\varepsilon$. Hence, $x$ is $\varepsilon$-stable. We conclude recalling that any nonsensitive subshift is almost equicontinuous.
\end{proof}

Actually, asymptotically periodic sofic subshifts are exactly those that have little simulation power, as suggested by the following proposition. We say that a subshift is \dfn{universal} if it can \ssl{simulate} all the other subshifts.
\begin{prop}\label{p:choix}
If $\Sigma$ is a sofic subshift, the following statements are equivalent.
\begin{enumerate}
\item\label{i:univs} $\Sigma$ is \ssl{not universal}.
\item\label{i:innum} $\Sigma$ is \ssl{countable}.
\item\label{i:ultper} $\Sigma$ is \ssl{asymptotically periodic}.
\item\label{i:inftrans} $\Sigma$ has no \ssl{infinite transitive subsystem}.
\item\label{i:infconnx} $\Sigma$ is the \ssl{label system} of some graph with no \ssl{non-cyclic strongly connected component}.
\end{enumerate}
\end{prop}
\begin{proof}~\begin{itemize}
\item[\ref{i:innum}$\impl$\ref{i:univs}:\ ]It is clear that a countable system cannot simulate an uncountable one (like a full shift on two letters).
\imp{ultper}{innum}If $\Sigma$ is asymptotically periodic, then we already noted that we can see it as the label system of a graph in which all the strongly connected components are cycles. Each configuration of $\Sigma$ has a path that changes of strongly connected component only a finite number of times. The tuple of the indices of the cells that correspond to these changes of component and of the corresponding arc determines in a unique way the configuration. $\Sigma$ is hence countable.
\imp{inftrans}{ultper}We have seen that $\omega_\Sigma$ is the union of the maximal transitive subshifts of $\Sigma$. Hence, if $\sigma\restr{\omega_\Sigma}$ is not periodic, then there is a transitive subsystem which is note periodic. On the other hand, it is known that transitive sofic subshifts are either cycles or infinite.
\imp{infconnx}{inftrans}We know that any transitive subsystem of a sofic subshift is exactly the set of labels of a strongly connected components of some corresponding graph.
\imp{univs}{infconnx}If $\Sigma$ is the label system of a graph $(V,E)$ with some non-cyclic strongly connected component, then there exists three vertices $v_0$, $v_1$ and $\tilde v_1$ in this component and two distinct letters $a$ and $b$ such that $(v_0,v_1,a),(v_0,\tilde v_1,b)\in V$. By strong connectivity, there exists two paths $(v_i,w_i,u_i)_{0\le i\le l}$ and $(\tilde v_i,\tilde w_i,\tilde u_i)_{0\le i\le k}$ of respective lengths $l,k\in\Ns$ such that $v_0=w_l=\tilde v_0=\tilde w_k=v_0$.
Let $u=(a\tilde u)^{\length{b\tilde v}}$ and $v=(b\tilde v)^{\length{a\tilde u}}$. We can see that $\dinf{(u+v)}$ is included in $\Sigma$. Moreover, it can be easily seen that $(\dinf{(u+v)},\sigma^{\length u})$ is conjugate to the full shift $(\deux^\Mset,\sigma)$.
The latter full shift is universal, since for any alphabet $B$, there is a trivial injection from $B$ into $A^{\spart{\log\card B}}$, which induces a conjugacy of any subshift over $B$ onto some subsystem of the iterate $(\deux^\Mset,\sigma^{\spart{\log\card B}})$.
\finish{proof}

Let us now concentrate on the case of asymptotic nilpotency, which in particular implies the previous consequences of asymptotic preperiodicity. Note that if all configurations of a sofic subshift converge towards the same configuration, then this configuration is uniform and the subshift can be seen as the label system of a graph in which all the cycles share the same label. More formally, a sofic subshift is \ssl{asymptotically nilpotent} if and only if it contains a \ssl{unic} \ssl{periodic} configuration, which is then \ssl{uniform}. Using some of the previous results, it can equivalently be said that a sofic subshift is \ssl{asymptotically nilpotent} if and only if it is the \ssl{label system} of a graph in which each \ssl{strongly connected component} is a single arc, and all of them have the same label.

Using the particular case $p=1$ in Proposition \ref{p:fnilpper2} gives us an equivalence between asymptotic nilpotency and weak nilpotency for onesided subshifts. This result in not true for general DDS.

The asymptotically nilpotent subshifts which are not nilpotent are actually rather complex. For instance, any \ssl{asymptotically nilpotent} subshift is an SFT if and only if it is \ssl{nilpotent}.

\subsection{Asymptotic set of cellular automata}\label{ss:ultca}

The intrinsic regularity of the model of CA allows more precise characterizations of asymptotic behaviors. We are going to present some of them, but maybe more are to be expected.
The first easy remark is that the asymptotic set of a CA is shift-invariant. However, unlike the limit set, it need not be a subshift (see below)  and can be arbitrarily complex (see for example \cite{acc}).

\begin{exmp}[Non-closed asymptotic set]
This example is due to Matthieu Sablik. Consider a CA with six states: particle going to the left, particle going to the right, wall, L, R, killer. A particle makes rebounds between walls, ensuring that it is the only particle between two walls, that the cells between the left wall and itself are in state L, and that the cells between itself and the right wall are in state R. If the configuration is not well formed, a killer state appears and spreads towards both sides. Any configuration with a singe particle between two walls, with L on its left and R on its right is actually periodic. If we look at a sequence of such configurations where the walls are further and further from the central cell -- still containing the same particle -- then it converges to the configuration where the particle is between only L on its left and only R on its right. It is easy to see that this configuration cannot be the adhering value of any orbit.
\end{exmp}

Of course, any \ssl{quiescent} configuration being the limit of its own orbit, it is in the \ssl{asymptotic set}. Moreover, since the uniform configurations constitute a subsystem, we can see that there is always at least one \ssl{uniform} configuration in the \ssl{asymptotic set}.

The homogeneity of the CA makes the space somehow so rigid that they satisfy an analogous of Poincaré's theorem.
\begin{thm}[Bernardi \cite{bernardi}]
A CA is \ssl{surjective} if and only if its set of \ssl{recurrent} configurations is dense.
\end{thm}

In particular, from Proposition \ref{p:recautotr}, a CA is \ssl{surjective} if and only if it is \ssl{nonwandering}. This fact was also proved via ergodic theory in \cite{topca}. Thus it is possible to link CA surjectivity and asymptotic set.
\begin{cor}
A CA is \ssl{surjective} if and only if its \ssl{asymptotic set} is a \ssl{residual set}.
\end{cor}

This characterization is not as strong as that we have on limit sets; one can wonder if an equivalent one could be found. Surprisingly, this simple question is still open.
\begin{open}
Does there exist a \ssl{surjective} CA whose \ssl{asymptotic set} is not \ssl{full}?
\end{open}

If the asymptotic set is a global notion, the trace is a local observation of the behavior. In this way, it is easy to think that if some global behavior (asymptotic periodicity, asymptotic nilpotency, \ldots) is true, then the local observation (\ie, the traces) will show the same behavior. The more interesting question is whether the converse is true or not. A first result of that kind is the following proposition.
 \begin{prop}
 If $p\in\Ns$, then a TDDS $(\Sigma,F)$ is \ssl{asymptotically $p$-periodic} if and only if all of its traces $\tau\iexp k_F$, for $k\in\Nset$, are \ssl{weakly $p$-preperiodic}.
\end{prop}
\begin{proof}
 Let $x\in\omega_F$, \ie there exists a configuration $y$ whose orbit $\orb_F(y)$ admits $x$ as an adhering value. If any trace is $p$-preperiodic, then for any $k\in\Nset$, there exists a generation $j\in\Nset$ such that $T_F\iexp kF^j(y)$ is $p$-periodic. By continuity of the trace, $T_F\iexp k(x)$ is $p$-periodic. Putting things together, $x$ is $p$-periodic.
 The converse comes from the preservation of the asymptotic periodicity by factor maps and from Proposition \ref{p:fnilpper2}.  
\end{proof}

For the specific case of PCA, the homogeneity of the rule imposes every trace to be preperiodic as soon as the trace of width $1$ is, and we have the following stronger statement.  \begin{prop}
  If $p\in\Ns$, then a PCA $F$ is \ssl{asymptotically $p$-periodic} if and only if its \ssl{trace} $\tau_F$ is \ssl{weakly $p$-preperiodic}.
\end{prop}

Nevertheless, the preperiod cannot be made uniform: the CA $\Min$, for instance, is asymptotically periodic but not weakly preperiodic. Even imposing a unique ultimate periodic word in each trace cannot help get a bounded preperiod, as illustrated by the following example.
\begin{exmp}[Non-preperiodic CA of weakly preperiodic trace]\label{x:p012}
Let $F$ bet he CA defined on alphabet $\trois$, with anchor $1$, diameter $4$ by the following local rule: \[\appl f{\trois^4}\trois{(x_{-1},x_0,x_1,x_2)}{\soit{x_{-1}+1\bmod3&\textrm{ if } x_{-1}\ne x_0\ne x_1=x_2~;\\x_0+1\bmod3&\textrm{ otherwise}.}}\]

 By recurrence, one can see that two consecutive cells in the same state will always keep an identical state. In particular, a cell that applies the first part of the rule gets the same state as its right neighbor, and both of them will never apply the second part of the rule after that. As a conclusion, $\tau_F\subset{\orb_\sigma((012)^*(02+12+01)\uinf(012))}$ is \ssl{weakly $3$-preperiodic}.  

 Let $\Phi$ the simulation by $(\Sigma_K\subset\trois^\Zset,\sigma^2)$ of $(\deux^\Nset,\sigma)$ defined:

\begin{eqnarray*}
 \textrm{ on subshift }\Sigma_K\textrm{ of forbidden language }&
 \ K=\bigcup\suba{a,b\in\trois\\k\in\Nset}aaA^{2k+1}bb\cup\{000,111,222\}\\
 \textrm{ by the local rule }&
 \ \phi:(x_0,x_1)\mapsto\soit{1&\textrm{ if } x_0\ne x_1~;\\0&\textrm{ otherwise}.}
\end{eqnarray*}
 From the definition of $f$, $\Phi$ is a conjugacy of $F\restr{\Sigma_K}$ into $\Min$. In particular, $F$ simulates a CA that is \ssl{not preperiodic}. Therefore, neither can $F$ be preperiodic.

\end{exmp}

Total disconnection allows a generalization of Proposition \ref{p:nilbloc}.
\begin{prop}\label{p:perbloc}
 No asymptotically periodic TDDS is sensitive.
\end{prop}
\begin{proof}
Suppose that $(\Sigma,F)$ is a sensitive asymptotically periodic TDDS. By Propositions \ref{p:senstr} and \ref{p:ultsim}, all of its traces sufficiently thin also have both properties, which contradicts Corollary \ref{c:fnilpper4}.
\end{proof}

Similarly to nilpotency and weak nilpotency, we can see, thanks to shift-invariance of the asymptotic set, that any configuration $z$ such that some CA is asymptotically $z$-nilpotent is uniform. We will speak of asymptotically $0$-nilpotent CA, where $0$ is a quiescent state of $A$.

We can use Proposition \ref{p:ultsim} to deduce that a TDDS $F$ is \ssl{asymptotically nilpotent} if and only if all of its traces are \ssl{weakly nilpotent}. This result can be simplified if $F$ is a PCA, since each projection of $\tau_F\iexp k$ coincides with the trace $\tau_F$. Hence $F$ is \ssl{asymptotically nilpotent} if and only if its \ssl{trace} $\tau_F$ is \ssl{weakly nilpotent}. We are going to prove that, in the case of one-dimensional CA, asymptotic nilpotency is a very strong property, equivalent to nilpotency. 
 \begin{lem}\label{l:chevauchements}
 Let $F$ a PCA on some twosided SFT $\Sigma$, such that for any generation $j\in\Nset$, there exists some \ssl{$0$-finite} \ssl{$0$-nilpotent} configuration $x\in\Sigma$ such that $F^j(x)\ne\dinf0$. Then $F$ is \ssl{not asymptotically $0$-nilpotent}.
\end{lem}
\begin{proof}
  Assume $F$ has radius $r\in\Nset$ and is asymptotically nilpotent.

 Let us first show that the configuration can be taken with \xpr{holes}, \ie for any $k\in\Nset$, there is a $0$-finite $0$-nilpotent configuration $x'\in[0\iexp k]$ and a generation $j>k$ such that $F^j(x')_0\ne0$. Indeed, by asymptotic nilpotency and compactness, there exists a generation $J\in\Nset$ such that $\forall x\in\Sigma,\exists j<J,F^j(x)\in[0\iexp k]$. By hypothesis, and maybe thanks to a composition by a shift, there exists a $0$-finite $0$-nilpotent configuration $x$ such that $F^{k+J}(x)_0\ne0$; hence there exists a configuration $x'=F^j(x)\in[0\iexp k]$, that is still $0$-finite and $0$-nilpotent (as are all the configurations of the orbit of $x$), such that $F^{k+J-j}(x')_0\ne0$, with $j<J$ and hence $k+J-j>k$. 

  Let us now show that if $x\in\Sigma$ is a $0$-finite $0$-nilpotent configuration and $k\in\Nset$, then there exists a $0$-finite $0$-nilpotent configuration $y\in[x\isub{rk}]$ such that $\set j\Nset{F^j(y)_0\ne0}\supsetneq\set j\Nset{F^j(x)_0\ne0}$. Indeed, $F^n(x)=\dinf0$ for some generation $n\in\Nset$. $k$ can be enlarged so that we can suppose that $x\in\pinf0[A^{2r(k-2n)}]\pinf0$. The previous point gives a $0$-finite $0$-nilpotent configuration $x'\in[0\iexp{2rk}]$ and a generation $j>k$ such that $F^j(x')_0\ne0$. $n$ can be enlarged so that one can assume that $\Sigma$ is a $2n$-SFT; consequently, it contains the configuration $y=x'\soo{-\infty}{-rk}[x\isub{rk}]x'\soo{rk}{\infty}$. By an immediate recurrence on generation $j\le n$, we can see that $F^j(y)_i=F^j(x')_i\textrm{ if }\abs i>r(k-2n-j)$ and $F^j(y)_i=F^j(x)_i\textrm{ if }\abs i\le r(k-2n+j)$. In particular: \[\set j\Nset{F^j(y)_0\ne0}\cap\cc0n=\set j\Nset{F^j(x)_0\ne0}\cap\cc0n=\set j\Nset{F^j(x)_0\ne0}~.\] On the other hand, since $F^n(x)\isub{r(k-n)}=0^{2r(k-n)}=F^n(x')\isub{r(k-n)}$, one can see that $F^n(y)=F^n(x')$. By construction, there is a generation $j\ge k>n$ such that $F^j(y)_0=F^j(x')_0\ne0$. As a result, $\set j\Nset{F^j(y)_0\ne0}\supsetneq\set j\Nset{F^j(x)_0\ne0}$. 

  Therefore, we can inductively build a sequence $(y^k)_{k\in\Nset}$ of $0$-finite $0$-nilpotent configurations, with $x^0=\dinf0$ and for any $k\in\Nset$, $x^{k+1}\in[x^k\isub{r(k+1)}]$ and $\set j\Nset{F^j(x^{k+1})_0\ne0}\supsetneq\set j\Nset{F^j(x^k)_0\ne0}$. This sequence converges towards the configuration $x\in\bigcap_{k\in\Nset}[x^k\isub{r(k+1)}]$, which is such that $\set j\Nset{F^j(x)_0\ne0}$ contains $\set j\Nset{F^j(y^k)_0\ne0}$ for any $k\in\Nset$ (by continuity of the trace application). This sequence of sets being strictly increasing, $\set j\Nset{F^j(x)_0\ne0}$ is infinite, \ie the trace $\tau_F$ is not weakly $0$-nilpotent.  
\end{proof}

With this lemma, we can prove the already mentioned theorem linking the global behavior and asymptotic behavior of the radius $1$ traces. This generalizes a result presented in \cite{nilpeng}.
\begin{thm}\label{t:fnilptr}
 Any \ssl{asymptotically $0$-nilpotent} PCA on a transitive SFT is \ssl{$0$-nilpotent}.
\end{thm}
\begin{proof} 
  Let $F$ be an asymptotically $0$-nilpotent PCA of radius $r\in\Nset$ over some transitive SFT $\Sigma\subset\am$, whose order $l$ can be assumed equal to $2r$ (enlarging $l$ or $r$ if need be). We can assume that $\Mset=\Zset$ without altering the properties of nilpotency and limit nilpotency.
Proposition \ref{p:nilbloc} and Theorem \ref{t:quasieq} give an $l$-blocking word $u$, \ie $T^l_F([u]_{-i})$ is a singleton for some $i\in\Zset$. There exists some generation $k\in\Nset$ such that $\forall n\ge k,\forall x\in[u]_{-i},F^n(x)\in[0^l]_0$. Let $j\in\Nset$. Suppose that $F$ is not nilpotent; it gives a configuration $x\in\az$ such that $F^{j+k}(x)_0\ne0$.  $\Sigma$ being transitive, it contains a configuration $x'=zuv[x\isub{r(j+k)}]v'uz'$, with $z\in\anm,z'\in\an,v,v'\in A^*$. This latter configuration has the property that, if $p_1=-r(j+k)-\length{uv}+i$ and $p_2= r(j+k)+\length{v'}+i$, then: \[\forall n\ge k,F^n(x')\sco{p_1}{p_1+l}=F^n(x')\sco{p_2}{p_2+l}=0^l~.\] $\Sigma$ being an $l$-SFT containing $\dinf0$ (as the limit of all the orbits of $F$), it also contains the configuration $y=\pinf0[F^k(x')\sco{p_1}{p_2+l}]_{p_1}\uinf0$. By construction, $F^j(y)_0=F^{j+k}(x)_0\ne0$. As the concatenation of parts of three configurations sharing the same traces of width $l$ in cells $p_1$ and $p_2$, one can see from Proposition \ref{p:concat} that for any generation $n\in\Nset$, $F^n(y)\soo{-\infty}{p_1+l}=\pinf0$ and $F^n(y)\sco{p_2}{\infty}=\uinf0$. Besides, asymptotic nilpotency gives a generation $n\in\Nset$ for which $F^n(y)\in[0^{p_2-p_1-l}]_{p_1+l}$; it results that $F^n(y)=\dinf0$. The configuration $y$ is $0$-finite, $0$-nilpotent, but dies arbitrarily late (after at least $j$ generations); this contradicts Lemma \ref{l:chevauchements}.
\end{proof}

Note that the juxtaposition of the blocking words is, as in the proof of Theorem \ref{t:quasieq}, the crucial point that prevents a direct generalization of the proof to higher dimensions, that would nevertheless seem natural (CA can be defined similarly on any grid).
\begin{conj}
 Any asymptotically nilpotent $d$-dimensional CA $(A^{\Mset^d},F)$, $d\in\Ns$, is nilpotent.
\end{conj}


In the case of the limit set, there was a clear dichotomy between nilpotency (the limit set is a singleton) and other cases (the limit set is infinite). Here, we cannot achieve such a dichotomy (think about the $\Min$ automaton which has an asymptotic set with only two singletons).  However, it is possible to achieve a similar result looking at finite configurations. Let us introduce variants up to a shift of already introduced notions: an $(F,\sigma)$-periodic configuration is a $F\sigma^k$-periodic configuration for some $k\in\Zset$. A jointly $F$-periodic configuration is an $F$-periodic and $\sigma$-periodic configuration. Let $\Jcal_F$ be the set of jointly $F$-periodic configurations.

If $F$ is a (twosided) PCA of radius $r$ over some subshift $\Sigma\subset\az$, and $0\in A$, we say that a configuration $x$ is $(F,0)$-separated with width $k\in\Ns$, time $J\in\Ns$ and shift $s\in\cc{-rJ}{rJ}$ if $x\sco{-2rJ}0=x\sco{k}{k+2rJ}=0^{2rJ}$ and $F^J(x)\sco{-rJ}{k+rJ}=x\sco{s-rJ}{s+k+rJ}\ne0^{k+2rJ}$.
To state our result on the form of the asymptotic set of non-nilpotent PCA, we previously need to prove the following lemmata: the first one uses the finite type condition to build new configurations inside the asymptotic set from some known specific ones, and the last one ensures that this is applyabe in any non-nilpotent PCA.

\begin{lem}\label{l:separwper}
 Let $F$ be a (twosided) PCA of radius $r>0$ on some SFT $\Sigma\subset\az$ of order $r$, containing some $(F,0)$-separated configuration.
Then $\Sigma$ contains some nonuniform jointly periodic configuration.
\end{lem}
\begin{proof}
 Let $x\in\Sigma$ be $(F,0)$-separated with width $k\in\Ns$, time $J\in\Ns$ and shift $s\in\cc{-rJ}{rJ}$.
 Let $y\in\az$ be the configuration of period $k+2rJ$ such that $y\sco0{k+2rJ}=x\sco0{k+2rJ}$.
Thanks to the separation, each pattern of width $2rJ$ appearing in $y$ also appears in $x$:
\[\forall i\in\Nset,y\sco i{i+2rJ}=x\sco{i\bmod(k+2rJ)}{i\bmod(k+2rJ)+2rJ}~.\]
It results that $y$ is in the SFT $\Sigma$ of order $r$.
Moreover, again by catenating patterns of width $2rJ$ from the definition, we can see that $F^J\sigma^s(y)=y$.
In particular, $F^{(k+2rJ)J}\sigma^{(k+2rJ)s}(y)=y$, but we also know by construction that $\sigma^{(k+2rJ)s}(y)=y$, hence $y$ is jointly periodic.
\end{proof}

We say that a subshift is nontrivial if it is not reduced to a single configuration.
\begin{lem}\label{l:separinf}
 Let $F$ be a surjective twosided PCA on some nontrivial subshift $\Sigma\in\az$.
Then $\Sigma$ contains either some $(F,0)$-separated configuration or some $0$-infinite configuration.
\end{lem}
\begin{proof}
If $0$ is not quiescent, then $\Sigma$ contains some uniform configuration distinct from $\dinf0$.
Otherwise, by surjectivity and nontriviality, there exist $(x^j)_{j\in\Zset}\in\Sigma^\Zset$ and $(s_j)_{j\in\Zset}$ with for any $j\in\Zset$, $\abs{s_j-s_{j-1}}\le r$, $x^{j+1}=F(x^j)$, and $\sigma^{s_j}(x^j)_0\ne0$.
Let $t_1=1$, $l_1=0$ and, for $k\in\Ns$, $t_{k+1}= (\card A^{2l_k}+1)t_k$ and $l_{k+1}={l_k+4rt_{k+1}}$.
For $k\in\Ns$, define $U_k=\set y\Sigma{y\sco{-l_k-2rt_{k+1}}{-l_k}\ne0^{2rt_{k+1}}\textrm{ or } y\sco{l_k}{l_k+2rt_{k+1}}\ne0^{2rt_{k+1}}}$.
Let us prove by recurrence on $k\in\Ns$ that $\forall J\in\Zset,\exists j\in\oc{J-t_k}J,\sigma^{s_j}(x^j)\in\bigcap_{1\le l<k}U_l$.

The case $k=1$ is trivial (no intersection).

Assume that $k\in\Ns$ is such that $\forall J\in\Zset,\exists j\in\oc{J-t_k}J,\sigma^{s_j}(x^j)\in\bigcap_{l<k}U_l$.
Applying this property ${\card A^{2l_k}+1}$ times, for any $J\in\Zset$, we can find some distinct $(j_m)_{0\le m\le{\card A^{2l_k}}}$ in $\oc{J-t_{k+1}}J$ such that $\sigma^{s_{j_m}}(x^{j_m})\in\bigcap_{l<k}U_l$ for any $m\in\cc0{\card A^{2l_k}}$.
Suppose that for any $m\in\cc0{\card A^{2l_k}}$, $\sigma^{s_{j_m}}(x^{j_m})\notin U_k$, \ie $\sigma^{s_{j_m}}(x^{j_m})\sco{-l_k-2rt_{k+1}}{-l_k}=\sigma^{s_{j_m}}(x^{j_m})\sco{l_k}{l_k+2rt_{k+1}}=0^{2rt_{k+1}}$.
By the pigeon-hole principle, there are some $m,m'\in\cc0{\card A^{2l_k}}$ with $m<m'$ such that $\sigma^{s_{j_m}}(x^{j_m})\sco{-l_k}{l_k}=\sigma^{s_{j_{m'}}}(x^{j_{m'}})\sco{-l_k}{l_k}$.
Hence $\sigma^{s_{j_m}-l_k}(x^{j_m})$ is $(F,0)$-separated with width $2l_k\in\Ns$, time $j_{m'}-j_m\le J$ and shift $s_{j_{m'}}-s_{j_m}$.

We have just proved that if there are no $(F,0)$-separated configurations, then for any $k\in\Ns$, $\forall J\in\Zset,\exists j\in\oc{J-t_k}J,\sigma^{s_j}(x^j)\in\bigcap_{l<k}U_l$.

By compactness, the closed intersection $\bigcap_{l\in\Nset}U_l$ is nonempty. By definition it contains some $0$-infinite configuration.
\end{proof}

\begin{prop}
  Let $F$ be a (twosided) PCA over some SFT $\Sigma\subset\az$ and $\Lambda\supset\Jcal_F$ a nontrivial strongly $F$-invariant subshift of $\Sigma$.
  Then $\Lambda$ contains some $0$-infinite configuration. 
\end{prop}
\begin{proof}
  $(\Lambda,F)$ is surjective, hence Lemma \ref{l:separinf} gives either some $0$-infinite configuration or some $(F,0)$-separated configuration. But if $(\Sigma,F)$ admits an $(F,0)$-separated configuration, then Lemma \ref{l:separwper} gives a nonuniform jointly periodic configuration, which is clearly infinite, and belongs in $\Lambda$ by hypothesis.
  \end{proof}
This proposition can be applied in particular to the closure of $\omega_F$, which satisfies all the hypotheses, as previously stated.

\begin{cor}
 A (twosided) CA $F$ over $\az$ is $0$-nilpotent if and only if $\cl{\omega_F}$\resp{$\Omega_F$} contains only $0$-finite configurations.
\end{cor}

Unlike the limit set, the asymptotic set of CA is very sensitive to shift compositions. Actually, shifting sufficiently a CA allows any limit configuration to become the adhering values of some orbit. We say that a DDS $(X,F)$ is \dfn{semitransitive towards $Y\subset X$} if for any nonempty open set $U\subset X$, any ball $\ball\varepsilon x$ of center $x\in Y$ and radius $\varepsilon>0$, and any generation $J\in\Nset$, there exists $j\ge J$ such that $F^j(U)\cap\ball\varepsilon x\ne\emptyset$. We can now generalize Proposition \ref{p:ulttrans} with this notion.
\begin{lem}\label{l:ultstrans}
 Let $(X,F)$ some DDS. The asymptotic set $\omega_F$ includes all the subsets $Y\subset X$ such that there is a subsystem $(X',F)$ which is semitransitive towards $Y$.
\end{lem}
\begin{proof}
 Let $x$ a point of such a subset $Y$, and $U_0=X'$. By induction, semitransitivity gives us sequences $(j_k)_{k\in\Ns}$ of integers and $(U_k)_{k\in\Ns}$ of open sets of $X'$ such that $j_k>k$, $F^{j_k}(U_k)\cap\ball{1/k}x\ne\emptyset$ and $U_{k+1}=U_k\cap F^{-j_k}(\ball{1/k}x)$. By compactness, the intersection $\bigcap_{k\in\Nset}U_k$ is nonempty, and any of its elements admits $x$ as orbit adhering value.
\end{proof}

\begin{prop}\label{p:omegomeg}
Let $F$ be some \ssl{oblic} CA. Then $\omega_F=\Omega_F$.
\end{prop}
\begin{proof}
By Lemma \ref{l:ultstrans}, it is enough to show that $(\az,F)$ is semitransitive towards $\Omega_F$.
Let $k,l\in\Nset$, $u\in A\iexp k$, and $x$ a configuration of the limit set, \ie for any generation $j\in\Nset$, there exists a configuration $x^j$ such that $F^j(x^j)=x$.
Assume $F$ has an anchor $m<0$, a diameter $d\in\Nset$, a local rule $f$, and that $k+1+l=jm$ for some $j\in\Nset$ (otherwise enlarge one of them).
It is then easy to see that any configuration $y\in[ux^j\soc k{l+(d-1-m)j}]_{-k}$ is in the open set $U=[u]$ and satisfies $F^j(y)\in[x\scc{-l}l]=\ball{2^{-l}}x$.

\end{proof}
Any local rule can be seen as that of an oblic CA; this brings the following restatement.
\begin{cor}\label{c:omegomegs}
If $F$ is a CA of anchor $m\in\Zset$ and anticipation $m'\in\Nset$, then $\Omega_F=\omega_{\sigma^kF}$ for any $k>m$ and any $k<-m'$.
\end{cor}

The previous statements allow smart constructions of some counter-examples.
\begin{exmp}[Asymptotic set strictly including the union of the transitive subsystems]\label{x:minpenche}
The CA $\sigma\Min$ has anchor $-1$, diameter $2$ and the same local rule as the $\Min$ CA, \ie:\[\appl f{\deux^2}\deux{(a,b)}{a\times b~.}\]
As an oblic CA, its asymptotic set is equal to its limit set: $\omega_{\sigma\Min}=\Omega_{\sigma\Min}=(\pinf0+\pinf1)1^*(\uinf0+\uinf1)$. It includes strictly the asymptotic set of the limit system: $\omega_{\sigma\Min\restr\Omega}=\omega_{\sigma\Min\restr\omega}=\omega_{\Min}=\{\dinf0,\dinf1\}$, which is also the union of the transitive subsystems.
\end{exmp}

This example brings the following questions: is $\omega_{F\restr\omega}$ always the union of the transitive subsystems? How can we understand the fact that these two sets only differ by isolated configurations (the Cantor-Bendixon derivative)? When we look at the action of the shift over the asymptotic set, do we always have, as in the $\Min$ case, a \xpr{minimum} asymptotic set, which corresponds to the asymptotic set of the limit system of all the shifted versions? 

\section{Conclusion}

We studied discrete-time dynamical systems with respect to their behavior in the (very) long term.
After several general remarks on the topological properties of their limit set and asymptotic set, we focused on particular systems: subshifts (especially sofic subshifts) and partial cellular automata (especially cellular automata). In these two cases, the limit set and the asymptotic set have a very specific structure. The homogeneity of the models makes the properties over the limit behavior to constrain the possible transient evolution.


The diagram below summarizes the main implications we proved (note that $\Omega$-nilpotent, $\omega$-nilpotent, $\Omega$-periodic and $\omega$-periodic stand respectively for limit-nilpotent, asymptotically nilpotent, limit-periodic and asymptotically periodic).
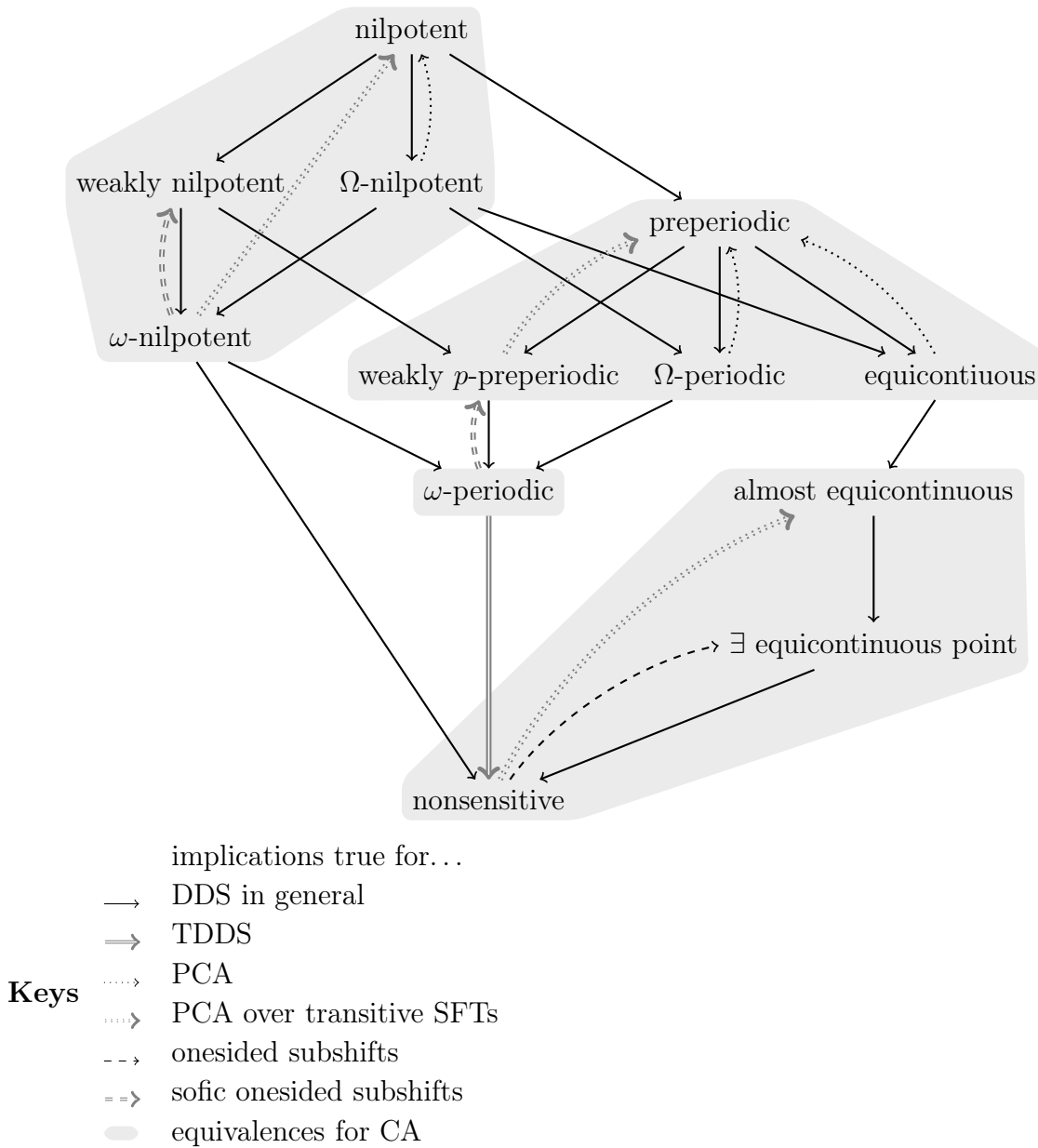
\begin{center}
%
%

\begin{tikzpicture}[thick=2,scale=1.1]
\node (nsens) at (0,0) {nonsensitive};
\node (qequi) at (5,2) {$\exists$ equicontinuous point};
\node (pequi) at (5,4) {almost equicontinuous};
\node (equi) at (6,5.5) {equicontiuous};
\node (aper) at (0,4) {$\omega$-periodic};
\node (anilp) at (-4,6) {$\omega$-nilpotent};
\node (fpper) at (0,5.5) {weakly $p$-preperiodic};
\node (lper) at (3,5.5) {$\Omega$-periodic};
\node (fnilp) at (-4,8) {weakly nilpotent};
\node (lnilp) at (-1,8) {$\Omega$-nilpotent};
\node (pper) at (3,7.5) {preperiodic};
\node (nilp) at (-1,10) {nilpotent};

\draw[eqCA] (anilp.south east) -- (lnilp.south east) -- (lnilp.north east) -- (nilp.north east) -- (nilp.north west) -- (fnilp.north west) -- (fnilp.south west) -- (anilp.south west) --cycle;
\draw[eqCA] (equi.south east) -- (equi.north east) -- (pper.north east) -- (pper.north west) -- (fpper.north west) -- (fpper.south west) --cycle;
\draw[eqCA] (nsens.south east) -- (qequi.south east) -- (pequi.south east) -- (pequi.north east) -- (pequi.north west) -- (nsens.north west) -- (nsens.south west) --cycle;
\draw[eqCA] (aper.south east) -- (aper.north east) -- (aper.north west) -- (aper.south west) --cycle;

\node (nsens) at (0,0) {nonsensitive};
\node (qequi) at (5,2) {$\exists$ equicontinuous point};
\node (pequi) at (5,4) {almost equicontinuous};
\node (equi) at (6,5.5) {equicontiuous};
\node (aper) at (0,4) {$\omega$-periodic};
\node (anilp) at (-4,6) {$\omega$-nilpotent};
\node (fpper) at (0,5.5) {weakly $p$-preperiodic};
\node (lper) at (3,5.5) {$\Omega$-periodic};
\node (fnilp) at (-4,8) {weakly nilpotent};
\node (lnilp) at (-1,8) {$\Omega$-nilpotent};
\node (pper) at (3,7.5) {preperiodic};
\node (nilp) at (-1,10) {nilpotent};

\draw[DDS] (nilp) -- (lnilp);
\draw[DDS] (nilp) -- (pper);
\draw[DDS] (nilp) -- (fnilp);
\draw[DDS] (fnilp) -- (anilp);
\draw[DDS] (fnilp) -- (fpper);
\draw[DDS] (lnilp) -- (anilp);
\draw[DDS] (lnilp) -- (lper);
\draw[DDS] (lnilp) -- (equi);
\draw[DDS] (pper) -- (fpper);
\draw[DDS] (pper) -- (lper);
\draw[DDS] (pper) -- (equi);
\draw[DDS] (anilp) -- (nsens);
\draw[DDS] (anilp) -- (aper);
\draw[DDS] (fpper) -- (aper);
\draw[DDS] (lper) -- (aper);
\draw[DDS] (equi) -- (pequi);
\draw[DDS] (pequi) -- (qequi);
\draw[DDS] (qequi) -- (nsens);

\path (aper) edge[TDDS] (nsens);
\path (equi) edge[PCA,bend right=22,looseness=.8] (pper);
\path (lper) edge[PCA,bend right=22,looseness=.8] (pper);
\path (fpper) edge[PCAt,bend left=22,looseness=.8] (pper);
\path (anilp) edge[PCAt] (nilp);
\path (lnilp) edge[PCA,bend right=22,looseness=.8] (nilp);
\path (nsens) edge[PCAt,bend left=22,looseness=.8] (pequi);
\path (nsens) edge[oSS,bend left=22,looseness=.8] (qequi);
\path (anilp) edge[soSS,bend left=22,looseness=.8] (fnilp);
\path (aper) edge[soSS,bend left=22,looseness=.8] (fpper);


\end{tikzpicture}
%

\end{center}
\textbf{Keys}
\begin{tabular}{ll}
&implications true for\ldots\\
 \tikz{\draw[DDS] (0,0) -- (.5,0);} &DDS in general\\
 \tikz{\draw[TDDS] (0,0) -- (.5,0);} &TDDS\\
 \tikz{\draw[PCA] (0,0) -- (.5,0);} &PCA\\
 \tikz{\draw[PCAt] (0,0) -- (.5,0);} &PCA over transitive SFTs\\
 \tikz{\draw[oSS] (0,0) -- (.5,0);} &onesided subshifts\\
 \tikz{\draw[soSS] (0,0) -- (.5,0);} &sofic onesided subshifts\\
 \tikz{\draw[eqCA] (0,-.1) rectangle +(.5,.2);}&equivalences for CA
\end{tabular}

\ack
This work was partially supported by the ANR project \textit{EMC} and the Academy of Finland project 131558.
We would also like to thank Guillaume Theyssier, Mathieu Sablik, François Blanchard, for very interesting discussions about the asymptotic set.

\section*{References}
\bibliographystyle{unsrt}
\bibliography{biblio_ult}

\end{document}